\documentclass[a4paper,10pt]{article}
\usepackage[english]{babel}
\usepackage{amsthm,amsmath,amssymb,bbm}
\usepackage{stackengine}
\usepackage{natbib}

\newtheorem{thm}{Theorem}[section]
\newtheorem{asu}[thm]{Assumption}

\newtheorem{lemma}[thm]{Lemma}

\newtheorem{prop}[thm]{Proposition}
\newtheorem{defin}[thm]{Definition}
\newtheorem{rema}[thm]{Remark}
\newtheorem{corollary}[thm]{Corollary}

\def \n{\Vert}

\def\E{{\mathbb{E}}}

\def\P{{\mathbb{P}}}
\def\R{{\mathbb{R}}}
\def\N{{\mathbb{N}}}

\def\F{{\cal{F}}}

\def\limt{\lim_{t\to\infty}}
\def\limt0{\lim_{t\to 0}}
\def\limn{\lim_{n\to\infty}}

\def\|{\,|\,}

\def\bn#1\en{\begin{align*}#1\end{align*}}
\def\bnn#1\enn{\begin{align}#1\end{align}}

\allowdisplaybreaks
\title{Stability in quadratic variation}
\author{Philip Kennerberg${}^*$ and Magnus Wiktorsson
\footnote{Institute of computing, Università della Svizzera italiana, Via la Santa 1, 6962 Lugano-Viganello Switzerland, philip.kennerberg@usi.ch}
\footnote{Centre for Mathematical Sciences, Lund University, Box 118 SE-22100, Lund, Sweden}}

\begin{document}
\maketitle

\begin{abstract}
Consider a sequence of cadlag processes $\{X^n\}_n$, and some fixed function $f$. If $f$ is continuous then under several modes of convergence $X^n\to X$ implies corresponding convergence of $f(X^n)\to f(X)$, due to continuous mapping. We study conditions (on $f$, $\{X^n\}_n$ and $X$) under which convergence of $X^n\to X$ implies $\left[f(X^n)-f(X)\right]\to 0$. While interesting in its own right, this also directly relates (through integration by parts and the Kunita-Watanabe inequality) to convergence of integrators in the sense $\int_0^t Y_{s-}df(X^n_s)\to\int_0^t Y_{s-}df(X_s)$. We use two different types of quadratic variations, weak sense and strong sense which our two main results deal with. For weak sense quadratic variations we show stability when $f\in 
C^1$, $\{X^n\}_n,X$ are Dirichlet processes defined as in \cite{NonCont} $X^n\xrightarrow{a.s.}X$, $[X^n-X]\xrightarrow{a.s.}0$ and $\{(X^n)^*_t\}_n$ is bounded in probability. For strong sense quadratic variations we are able to relax the conditions on $f$ to being the primitive function of a cadlag function but with the additional assumption on $X$, that the continuous and discontinuous parts of $X$ are independent stochastic processes (this assumption is not imposed on $\{X^n\}_n$ however), and $\{X^n\}_n,X$ are Dirichlet processes with quadratic variations along any stopping time refining sequence. To prove the result regarding strong sense quadratic variation we prove a new It\^o decomposition for this setting.
\end{abstract}

\section{Introduction}
Stability in quadratic variation, namely how $[f(X^n)-f(X)]\to 0$ when $X^n\to X$ for certain cadlag processes $\{X^n\}_n$, $X$ is intimately tied to convergence of $\int Y_-df(X^n)\to \int Y_-df(X)$ by the integration by parts and the Kunita-Watanabe inequality. Under assumptions on uniform convergence and that $f\in C^1$ results can be readily achieved by elementary means. Our focus in this manuscript is then to relax both of these conditions, i.e. assume pointwise convergence of $X^n\to X$ instead of uniform and relax the $C^1$ assumption on $f$. Most results will be concerning Dirichlet processes, although some more elementary result concern general cadlag processes with quadratic variation. We consider two types of quadratic variations, one which we term weak-sense is the type of quadratic variation that has typically used in the study of Dirichlet processes. This is indeed a weak type of quadratic variation as it only assumes that the limit (in probability) of the squared differentials only exists along a certain refining sequence (a sequence of fixed-time partitions with mesh tending to zero). The second kind of quadratic variation, we term quadratic variation in the strong sense assumes that approximating sum of squared differentials exists as uniform limit in probability along \textit{any} all sequences of stopping time partitions with mesh tending to zero. It should be pointed out that if $f$ is a transform that preserves the presence of quadratic variation for a process $X$ along its given fixed time refining sequence, then if $X$ has quadratic variation along all fixed time refining sequences so will $f(X)$. This however has no bearing on stopping time refining sequences.   
\\
\\
Continuous Dirichlet processes were first defined in \cite{Fol}, as a sum of a continuous local martingale and a continuous process with zero quadratic variation (along some given sequence of refining partitions). Discontinuous Dirichlet processes were later defined in \cite{Stricker} as the sum of a semimartingale and a process of zero quadratic variation (along some sequence of refining partitions). In \cite{NonCont} it was shown that a Dirichlet process transformed by a $C^1$ map is still a Dirichlet process. \cite{Low}, working with a different definition of Dirichlet processes (assuming strong quadratic variations) than previous authors showed that this space of processes was closed under a wide family of locally Lipschitz continuous functions and provided an It\^o formula for such transforms.

The results of \cite{Low} and \cite{NonCont} are quite different, both in the sense that they use different notions of quadratic variation and the fact that the decompositions are quite different. The decomposition in \cite{NonCont} gives a more detailed description of the non-continuous part. In the study of our problem, decompositions of the kind proved in \cite{NonCont} will be indispensable. We shall develop an analogous It\^o decomposition to this one, but in the setting of strong quadratic variation and with relaxed assumptions on the transform, $f$. It turns out however that one of the key components will be the results of \cite{Low} for proving this formula. In the setting of weak quadratic variation we show stability under the assumption of pointwise convergence $X^n\to X$ and $f\in C^1$. In the setting of strong quadratic variation we prove stability with a relaxed assumption on $f$ (being a primitive function of a cadlag function), pointwise convergence of $X^n\to X$, but some additional assumptions on $X$ ($X_s$ having a non-atomic distribution for each $s\in [0,t]$, and the continuous part and discontinuous part of $X$ being independent as stochastic processes).

\section{Preliminaries}
We assume that all of our processes are defined on a common filtered probability space $(\Omega,\mathcal{F},\{\mathcal{F}_t\}_{t \ge 0},\P)$ and that all the defined processes are adapted to the same filtration $\{\mathcal{F}_t\}_{t \ge 0}$ fulfilling the usual hypothesis.
The term refining sequence will refer to a sequence of partitions of $[0,t]$, $\{D_k\}_k$ say, such that $\lim_{k\to\infty}\max_{t_i\in D_k} |t_{i+1}-t_i|=0$. We say that a cadlag process $X$ admits to a quadratic variation $[X]$, in the weak sense, if there exists an increasing continuous process $[X]^c$ such that 
\begin{align}\label{quad}
[X]_s=[X]^c_s+\sum_{u\le s}(\Delta X_u)^2,
\end{align}
for every $0< s\le t$, there exists at least one refining sequence $\{D_k\}_k$ such that if we let  

$$(S_n(X))_s:=\sum_{t_i\in D_n,t_i \le s}\left( X_{t_{i+1}}-X_{t_i} \right)^2,$$ 
for every $0< s\le t$ (where we use the convention that $t_{i+1}=s$ if $t_i=s$) then 
\begin{align}\label{partialsums}
(S_n(X))_s\xrightarrow{\P}[X]_s\textit{ as }n\to\infty.
\end{align}
We say that two cadlag processes admit to a covariation $[X,Y]$ along $\{D_n\}_n$ if $[X,Y]^c$ is a continuous finite variation process such that
$$[X,Y]_s=[X,Y]^c_s+\sum_{u\le s}\Delta X_u\Delta Y_u,$$
for every $0< s\le t$ and if we let

$$S_n(X,Y)_s:=\sum_{t_i\in D_n,t_i \le s}\left( X_{t_{i+1}}-X_{t_i} \right)\left( Y_{t_{i+1}}-Y_{t_i} \right),$$ 
for every $0< s\le t$ then $S_n(X,Y)_s\xrightarrow{\P}[X,Y]_s$ as $n\to\infty$.
\\
\begin{defin}
  As in \cite{NonCont} we define a (non-continuous) Dirichlet process $X$ as a sum of a semimartingale $Z$ and an adapted continuous process $C$ with zero quadratic variation along some refining sequence.
\end{defin}
By this definition we see that Dirichlet processes are a subclass of the processes admitting to quadratic variation. Note however that this deomposition is not unique since it is possible to move any continuous part with finite variation between $Z$ and $C$.
Given a cadlag process $X_t$ with $t\ge 0$ and a stopping time $T$ we defined $X$ stopped at $T$ as $X^T_t=X_{t\wedge T}$, we also define the supremum process of $X$ as $X^*_t=\sup_{s\le t} |X_s|$.
\begin{defin}\label{loc}
  A property of a stochastic process is said to hold locally (pre-locally) if there exist a sequence of stopping times  $T_k$ increasing to infinity such that the property holds for the stopped process $X^{T_k}$ ($X^{T_k-}$) for each $k$.
\end{defin}
Let $\sigma_1$ and $\sigma_2$ be stopping times, a finite set of stopping times $P=\{\tau_0,...,\tau_n\}$ with $\sigma_1=\tau_0\le \tau_1\le ...\le \tau_n=\sigma_2$ is called a partition of $[\sigma_1,\sigma_2]$ and we define the mesh of $P$, as $\n P\n=\max_{1\le i\le n} \tau_i-\tau_{i-1}$. Given two cadlag processes, $X$ and $Y$, we denote $[X,Y]^{P}_s=\sum_{i=1}^n\left(X_{\tau_i\wedge s}-X_{\tau_{i-1}\wedge s}\right)\left(Y_{\tau_i\wedge s}-Y_{\tau_{i-1}\wedge s}\right)$ and $[X]^P_s=[X,X]^P_s$. 
\begin{defin}\label{partidef}
We say that $X$ and $Y$ admits to a covariation in the strong sense if there exists a FV process $[X,Y]$ such that 
$$\limsup_{\n P\n\le \delta, P\in \mathcal{P}_{[0,t]},\delta\to 0} \P\left(\left([X,Y]^P-[X,Y]\right)^*_t>\epsilon\right)=0,$$ 
for all $\epsilon>0$, where $\mathcal{P}_{[a,b]}$ denotes the set of partitions of $[a,b]$ (assuming $a\ge 0$ and $b\le t$). We also say that $X$ admits to a quadratic variation in the strong sense if $X$ and $X$ admits to a covariation in the strong sense.
\end{defin}
The above definition is equivalent to saying that for any sequence of partitions $\{P^n\}_n$ of $[0,t]$, such that $\n P^n \n\to 0$ a.s. as $n\to \infty$, we have $\lim_{n\to\infty} \P\left(\left([X,Y]^P-[X,Y]\right)^*_t>\epsilon\right)=0$ (this follows directly by proof by contradiction). We shall use both these equivalent properties but we prefer to define it as in \ref{partidef} as the $\limsup$ always exists (but might not be zero). 
\\
\\
\begin{defin}\label{calA}
Given a cadlag process $X$ with jump measure $\mu$ and compensation measure $\nu$ we shall let $\mathcal{A}=\left\{(s,\omega): s\le t,\nu\left(\omega,\{s\},\R\right)>0\right\}$, i.e. the set of purely predictable jumps for $X$.
\end{defin}
The following is a (special case) of Proposition 1.17 b) in \cite{JACOD}
\begin{prop}\label{version}
There exists a version of $\nu$ such that $\nu(\{s\},\R)\le 1$, for $s\le t$ and $\mathcal{A}$ is a thin set exhausted by a sequence of predictable times
\end{prop}
In this paper we shall always assume that we choose a version of $\nu$ according to Proposition \ref{version}.
\begin{lemma}\label{triangle}
Suppose $X^1,...,X^n$ are processes with quadratic variations along the same refining sequence and assume that $\sum_{k=1}^mX^k$ also has quadratic variation along this refining sequence for all $m\le n$ then
$$\left[\sum_{k=1}^nX^k\right]_t\le \left( \sum_{k=1}^n[X^k]_t^{\frac 12}\right)^2. $$
\end{lemma}
\begin{proof}
First of all we note that since $\sum_{k=1}^{m}X^k$ has a quadratic variation along the same refining sequence for each $m\le n$ it follows that the covariation $\left[\sum_{k=1}^{m}X^k,X^{m+1}\right]$ also exists along the same refining sequence. We will now prove the stated inequality by induction, the case $n=1$ is trivial. Assume the statement is true for $n=m$ then for $n=m+1$,
\begin{align*}
&\left[\sum_{k=1}^{m+1}X^k\right]_t =\left[\sum_{k=1}^{m}X^k\right]_t +2\left[\sum_{k=1}^{m}X^k,X^{m+1}\right]_t + \left[X^{m+1}\right]_t
\le \left[\sum_{k=1}^{m}X^k\right]_t +
\\
&2\left[\sum_{k=1}^{m}X^k\right]_t^{\frac 12}\left[X^{m+1}\right]^{\frac 12}_t + \left[X^{m+1}\right]_t=\left( \left[\sum_{k=1}^{m}X^k\right]^{\frac 12}_t +\left[X^{m+1}\right]^{\frac 12}_t\right)^2
\le \left(\sum_{k=1}^{m} [X_k]_t^{\frac 12} +\left[X^{m+1}\right]^{\frac 12}_t \right)^2
\\
&=\left(\sum_{k=1}^{m+1} [X_k]_t^{\frac 12} \right)^2,
\end{align*}
where we used the Kunita-Watanabe inequality in the second step.
\end{proof}

\section{Stability in quadratic variation}
If a sequence of cadlag processes $\{X^n\}_n$ converges in some sense to another cadlag process $X$ while we also have convergence of $\left[f(X^n)-f(X)\right]\to 0$, for some function $f$, we informally refer to this (as we have not specified in what sense this converges) as stability in quadratic variation. While we feel that stability in quadratic variation, in the sense we will explore, is interesting in its own right, it is also intimately tied to stability of integrators. One might be concerned with integrals of the form $\int Ydf(X^n)$, for some continuous function $f$ and wish to study the convergence of this integral as $X^n\to X$ in some sense. Due to the integration by parts formula
$$\int_0^t Y_{s-}df(X_s)=f(X_t)Y_t-f(X_0)Y_0-\int_0^tf(X_{s-})dY_s-[f(X),Y]_t $$
and by the Kunita-Watanabe inequality we have that if $[f(X)-f(X^n)]_t\to 0$ then $\int_0^t Y_{s-}df(X^n_s)$ converges to $\int_0^t Y_{s-}df(X_s)$ if and only if $\int_0^tf(X^n_{s-})dY_s\to\int_0^tf(X_{s-})dY_s$.
\\
Let us first consider a quite useful but elementary result regarding stability in quadratic variation.
The first very elementary result can be proven directly from first principles.
\begin{prop}
Suppose $\{X^n\}_n$ and $X$ have quadratic variations in either the weak (or the strong sense), $f\in C^1$, $(X^n-X)^*_t\to 0$ in probability and $[X^n-X]_t\to 0$ in probability. Then $[f(X^n)-f(X)]_t\to 0$ or in probability, with the quadratic variation being in the weak (strong) sense.
\end{prop}
\begin{proof}
We the proof for weak sense quadratic variations, the strong sense case is analogous. By localization, we may without loss of generality assume that $X^*_t\le R$ for some $R\in\R^+$. Let $\delta>0$ be so small that if $x,y\in [-(R+\delta),R+\delta]$ and $|x-y|<\delta$ then $|f(x)-f(y)|<\epsilon$. Let $N$ be so large that $n\ge N$ implies that $[X^n-X]_t\le \epsilon$ and $(X^n-X)^*_t<\delta$. For any partition of $[0,t]$ such that $\max_i |t_i-t_{i-1}|<\delta$ and $n\ge N$ we have, 
\begin{align*}
&\sum_{i}\left|f(X_{t_i})-f(X^n_{t_i})-\left(f(X_{t_{i-1}})-f(X^n_{t_{i-1}})\right)\right|^2
\\
=
&\sum_{i}\left|\int_0^1 \left(f'\left(X_{t_i}+\theta\left(X_{t_i}-X_{t_{i-1}}\right)\right)\left(X_{t_i}-X_{t_{i-1}}\right)-f'\left(X^n_{t_i}+\theta\left(X^n_{t_i}-X^n_{t_{i-1}}\right)\right)\left(X^n_{t_i}-X^n_{t_{i-1}}\right)\right)d\theta\right|^2
\\
\le
&2\sum_{i}\left|\int_0^1 f'\left(X_{t_i}+\theta\left(X_{t_i}-X_{t_{i-1}}\right)\right)d\theta\left(X_{t_i}-X^n_{t_i}-(X_{t_{i-1}}-X^n_{t_{i-1}})\right)\right|^2
\\
+
&2\sum_{i}\left|\int_0^1 \left(f'\left(X_{t_i}+\theta\left(X_{t_i}-X_{t_{i-1}}\right)\right)-f'\left(X^n_{t_i}+\theta\left(X^n_{t_i}-X^n_{t_{i-1}}\right)\right)\right)d\theta\left(X^n_{t_i}-X^n_{t_{i-1}}\right)\right|^2
\\
\le
&2\sum_{i}M^2\left| X_{t_i}-X^n_{t_i}-(X_{t_{i-1}}-X^n_{t_{i-1}})\right|^2
+
2\sum_{i}\epsilon^2\left|X^n_{t_i}-X^n_{t_{i-1}}\right|^2.
\end{align*}
If we now let the mesh of the partition tend to zero in the left-most and right-most side above have have the following limits in probability
\begin{align*}
\P\left([f(X)-f(X^n)]_t\ge a\right)
&\le
\P\left(2M^2\left(\left[X-X^n\right]_t +\epsilon \right)+2\epsilon^2[X^n]_t \ge a\right)
\\
&\le
\P\left(2M^2\epsilon+4\epsilon^2[X]_t+[X-X^n]_t\left(2M^2+4\epsilon^2\right) \ge a\right),
\end{align*}
for any $a>0$. Letting $\epsilon\to 0$ we have $[f(X)-f(X^n)]_t
\le
2M^2[X-X^n]_t$ which converges to zero as $n\to\infty$.
\end{proof}
Due to the fact that the above argument makes explicit use of refining sequences (which converge in probability) we are tied to convergence in probability. In section \ref{main} we aim to find situations when we may improve to almost sure convergence. We shall also relax the conditions $(X^n-X)^*_t\to 0$ and/or $f\in C^1$. It turns out that in many (but a bit less general in terms of what processes it applies to) situations it suffices that $X^n_s\to X_s$ for $s\le t$.
\\
\\
The second stability result, while more elementary than our main results has a more elaborate proof, but it is one that is more or less just a simplified version of those. This result relaxes the condition of uniform convergence of $X^n\to X$ on compacts to point-wise convergence while requiring a stronger condition on $f$, $f\in C^2$. We are also able to improve the convergence to almost sure convergence since we are not relying on arguments making explicit use of the refining sequences. Key to this result is the following decomposition due to \cite{Fol}.
\begin{lemma}\label{OGFollm}
Suppose $X$ is a process in with quadratic variation in the weak sense, $f$ is a $C^2$ function and $\{D_k\}_k$ is a refining sequence then
\begin{align}\label{FolDe}
f(X_t)=f(X_0)+\sum_{s\le t}\left(\Delta f(X_s)-\Delta X_sf'(X_{s_-}) \right)+\int_0^tf'(X_{s_-})dX_s+\frac 12\int_0^tf''(X_{s_-})d[X]^c_s,
\end{align}
where $\int_0^tf'(X_{s_-})dX_s$ is the point-wise limit of the series $\sum_{t_i\in D_k}f'(X_{t_{i-1}})(X_{t_{i}}-X_{t_{i-1}})$.
\end{lemma}
\begin{prop}\label{PropVar}
Suppose $X$  and $\{X^n\}_n$ are processes admitting to quadratic variations along the same refining sequence, $[X^n-X]_t\to 0$ a.s., $X^n_s\to X_s$ a.s., for $s\le t$, $\{(X^n)^*_t\}_n$ is bounded in probability and $f$ is a $C^2$ function then
$$[f(X^n)-f(X)]_t \to 0 \hspace{2mm} a.s..$$
\end{prop}
Due to partial overlap with the proof of Theorem \ref{a.s.ettf}, we post-pone the proof to the end of this manuscript.

\section{It\^o Decompositions and lemmas}
In this section we gather the necessary tools for our main results. These consists of our It\^o decomposition and some lemmas regarding independence of stochastic processes.
The following assumption was originally used by \cite{Low} and since we will use one of his results we will need it.
%

\begin{asu}\label{aslow}
Suppose that $f$ is a locally Lipschitz continuous function and $X$ is a dirichlet process such that
\begin{align}\label{lbllow}
\int_0^t 1_{\{s:X_s\not\in \mathsf{diff}(f)\}}d[X]^c_s=0 \hspace{2mm}a.s.,
\end{align}
where $\mathsf{diff}(f)$ is the set of points where $f$ is differentiable
\end{asu}




Throughout the remainder of this paper, whenever we write $f'(x)$ we will be referring to the right-hand derivative of $f$, $f'(x)=\lim_{h\to 0^+} \frac{f(x+h)-f(x)}{h}$, if $f$ is in-fact differentiable at $x$ then this becomes, as we all know, the usual derivative. The following Lemma is an application of (a slight modification of) Theorem 2.1 in \cite{Low}.
\begin{lemma}\label{LowLemma}
Suppose that $f$ is a locally Lipschitz continuous function that has right-hand derivatives everywhere and $X$ is a Dirichlet process such that they satisfy Assumption \ref{aslow} then
$$[f(X)]_s=\int_0^s f'(X_{u-})^2d[X]^c_s+\sum_{u\le s} \Delta f(X_u)^2.$$
\end{lemma}
\begin{proof}
We first note that Theorem 2.1. in \cite{Low} works just as well if we assume that $f$ has pointwise right-hand derivatives everywhere and use this as a definition of $D_x$ in place of the supremum-derivative (in Lemma 2.7. we merely let $\xi$ be be the corresponding limit for the right-derivative instead of a supremum). According to Theorem 2.1. in \cite{Low} we have that if $X$ is sum of a semimartingale and a process admitting a quadratic variation, with zero continuous quadratic variation then
\begin{align}\label{loweq}
f(X_s)=\int_0^sf'(X_{u-})dZ_u+V_s,
\end{align} 
where $V_s$ is a cadlag process that has a quadratic variation and zero continuous quadratic variation. Since $X=Z+C$ and $[C]^c=[C]=0$, $X$ does indeed satisfy this assumption. Note that since $f(X)$ has a quadratic variation, we must have that $[f(X)]_s=[f(X)]^c_s+\sum_{u\le s} \Delta f(X_u)^2$. On the other hand, if we compute the quadratic variation of \ref{loweq} we get
\begin{align}
[f(X)]_s=\int_0^s f'(X_{u-})^2d[Z]^c_s+\sum_{u\le s}  f'(X_{u-})^2(\Delta Z_u)^2+2\sum_{u\le s}  f'(X_{u-})^2\Delta Z_u\Delta V_u+\sum_{u\le s} \Delta V_u^2
\end{align}
and therefore we may conclude that $[f(X)]^c_s=\int_0^s f'(X_{u-})^2d[Z]^c_s=\int_0^s f'(X_{u-})^2d[X]^c_s$ from which the result follows.
\end{proof}

\begin{lemma}\label{chainrule}
If $C$ is a continuous process with zero quadratic variation in the strong sense, $Z$ is a semimartingale and $g$ is a continuous function then $\int_0^.Z_-g(C)dC$ is a well-defined ucp limit of its approximating Riemann-Stieltjes sums and it has zero quadratic variation in the strong sense.
\end{lemma}
\begin{proof}
First we establish the fact that if $C'$ is a process with strong quadratic variation and $Z'$ is a semimartingale then the integral $\int Z'_-dC'$ exists as the ucp limit of
$$ \sum_{\tau_i\in P^n}Z'_{\tau_{i-1}}\left(C'_{\tau_i}-C'_{\tau_{i-1}}\right),$$
for any sequence of partitions $\{P_n\}_n$ on $[0,t]$ with $\n P_n\n\to 0$. Moreover $\int_0^.Z_-dC'$ has zero quadratic variation in the strong sense. By the usual re-arrangement of the Riemann sums for integration by parts formula, we have
\begin{align*}
\sum_{\tau_i\in P^n}Z'_{\tau_{i-1}}\left(C'_{\tau_i}-C'_{\tau_{i-1}}\right)
=Z'C'-
 \sum_{\tau_i\in P^n}\left(Z'_{\tau_i}-Z'_{\tau_{i-1}}\right)\left(C'_{\tau_i}-C'_{\tau_{i-1}}\right) -\sum_{\tau_i\in P^n}C'_{\tau_{i-1}}\left(Z'_{\tau_i}-Z'_{\tau_{i-1}}\right)
\end{align*}
which converges in ucp to $Z'C'-[Z',C']-\int_0^.C'_-dZ'$, so this left-hand side is well-defined and a.s. finite and we therefore have
$$\int_0^sZ'_{u-}dC'_u=Z'C'-[Z',C']-\int_0^sC'_udZ'_u .$$
We will now show that $\int Z'_-dC'$ has zero quadratic variation.
By pre-local stopping we may assume $Z'\le L$ on $[0,t]$. Given any sequence of partitions of $[0,t]$, $\{P^n\}_n$ with $P^n=\{\tau_1^n,...,\tau_{|P^n|}^n\}$, such that $\n P^n\n\to 0$ as $n\to\infty$, we have
\begin{align*}
\left(\sum_{\tau_i^n\le .}\left(\int_{0}^{\tau_i} Z'_{s-}dC'_s-\int_{0}^{\tau_{i-1}} Z'_{s-}dC'_s\right)^2\right)^*_t
&=\left(\sum_{\tau_i^n\le .}\left(C'_{\tau_i}Z'_{\tau_i}-\int_{\tau_{i-1}}^{\tau_i} C'_sdZ'_s-C'_{\tau_{i-1}}Z'_{\tau_{i-1}}\right)^2\right)^*_t
\\
&=\left(\sum_{\tau_i^n\le .}\left(Z'_{\tau_i}(C'_{\tau_i}-C'_{\tau_{i-1}})-\int_{\tau_{i-1}}^{\tau_i} (C'_s-C'_{\tau_{i-1}})dZ'_s\right)^2\right)^*_t
\\
&\le
2L\left(\sum_{\tau_i^n\le .}\left(C'_{\tau_i}-C'_{\tau_{i-1}}\right)^2\right)^*_t
+
2\left(\sum_{\tau_i^n\le .}\left(\int_{\tau_{i-1}}^{\tau_i} (C'_s-C'_{\tau_{i-1}})dZ'_s\right)^2\right)^*_t
\\
&\le
2L\left(\sum_{\tau_i^n\le .}\left(C'_{\tau_i}-C'_{\tau_{i-1}}\right)^2\right)^*_t
+
4\left(\sum_{\tau_i^n\le .}\left(\int_{\tau_{i-1}}^{\tau_i} (C'_s-C'_{\tau_{i-1}})dV_s\right)^2\right)^*_t
\\
&+
4\left(\sum_{\tau_i^n\le .}\left(\int_{\tau_{i-1}}^{\tau_i} (C'_s-C'_{\tau_{i-1}})dM_s\right)^2\right)^*_t
\end{align*}
The term $\left(\sum_{\tau_i^n\le .}\left(C'_{\tau_i}-C'_{\tau_{i-1}}\right)^2\right)^*_t$ converges in probability to $[C']_t=0$. By pre-localization we may assume $[M]_t\le L_1$ and $Var(V)\le L_2$ (where $Var(.)$ denotes the total variation on $[0,t]$). By letting  $C'^n_s=\sum_{i=1}^{|P^n|}C'_{\tau^n_{i-1}}1_{[\tau^n_{i-1},\tau^n_i)}$ we have,
$$ \left(\sum_{\tau_i^n\le .}\left(\int_{\tau_{i-1}}^{\tau_i} (C'_s-C'_{\tau_{i-1}})dV_s\right)^2\right)^*_t
\le 
\sum_{i=1}^{|P^n|}\left(\int_{\tau_{i-1}}^{\tau_i} (C'_s-C'_{\tau_{i-1}})dV_s\right)^2
\le 
|P^n|L_2^24\left((C'-C'^n)^*_t\right)^2,$$
which converges to zero in probability. Let $\tilde{\tau}^n_i=\tau^n_i$ for $i=1,...,|P^n|$ and for $i>|P^n|$ we let $\tilde{\tau}^n_i=\tau^n_{|P^n|}$ (this will just be used to exchange a sum with a random number of elements for a deterministic series). By stopping we may assume $(C')^*_t\le L_3$. Due to Corollary 3 of chapter II.6 of \cite{PRT} ,
\begin{align}\label{dM}
\E\left[\left(\sum_{\tau_i^n\le .}\left(\int_{\tau_{i-1}}^{\tau_i} (C'_s-C'_{\tau_{i-1}})dM_s\right)^2\right)^*_t\right]
&=
\E\left[\sum_{i=1}^\infty\left(\int_{\tilde{\tau}^n_{i-1}}^{\tilde{\tau}^n_i} (C'_s-C'_{\tilde{\tau}^n_{i-1}})dM_s\right)^2\right]\nonumber
\\
&=\sum_{i=1}^\infty\E\left[\left(\int_{\tilde{\tau}^n_{i-1}}^{\tilde{\tau}^n_i} (C'_s-C'_{\tilde{\tau}^n_{i-1}})dM_s\right)^2\right]\nonumber
\\
&=\sum_{i=1}^\infty \E\left[\left[\int_{\tilde{\tau}^n_{i-1}}^{.} (C'_s-C'_{\tilde{\tau}^n_{i-1}})dM_s\right]_{\tilde{\tau}^n_i}\right]\nonumber
\\
&=
\sum_{i=1}^\infty\E\left[\int_{\tilde{\tau}^n_{i-1}}^{\tilde{\tau}^n_i} (C'_s-C'_{\tilde{\tau}^n_{i-1}})^2d[M]_s\right]\nonumber
\\
&\le
\E\left[4\left(\left( C'-C'^n\right)^*_t\right)^2[M]_t\right]
\end{align}
As 
$$4\left(\left( C'-C'^n\right)^*_t\right)^2[M]_t
\le 4\left(2L_3\right)^2L_1=16L_3^2L_1, $$
we may apply the dominated convergence theorem to conclude that 
$$\lim_{n\to\infty}\E\left[\left(\sum_{\tau_i^n\le .}\left(\int_{\tau_{i-1}}^{\tau_i} (C'_s-C'_{\tau_{i-1}})dM_s\right)^2\right)^*_t\right]=0$$
and therefore we also have convergence in probability by the Markov inequality. This shows that $\int Z'_-dC'$ indeed has zero quadratic variation.
\\
By localization we may assume $\sup_{s\in[0,t]}|g\left(C_s\right)|\le L_4$, for some $L_4\in\R^+$. Let $G(x)=g(0)+\int_0^xg(u)du$, then $G(C)$ has zero quadratic variation in the strong sense. Indeed,
\begin{align*}
\left(\sum_{\tau^n_i\in P^n,\tau^n_i\le . }\left(G(C_{\tau_i})-G(C_{\tau_{i-1}})\right)^2\right)^*_t
&=
\left(\sum_{\tau^n_i\in P^n,\tau^n_i\le . }\left(\int_0^1 g(C_{\tau_{i-1}}+\theta(C_{\tau_{i}}-C_{\tau_{i-1}}))d\theta(C_{\tau_{i}}-C_{\tau_{i-1}})\right)^2\right)^*_t
\\
&\le
\left(\sum_{\tau^n_i\in P^n,\tau^n_i\le . }\left(\int_0^1 \left|g(C_{\tau_{i-1}}+\theta(C_{\tau_{i}}-C_{\tau_{i-1}}))\right|d\theta\right)^2\left(C_{\tau_{i}}-C_{\tau_{i-1}}\right)^2\right)^*_t
\\
&\le
L_4^2\sum_{\tau^n_i\in P^n}\left(C_{\tau_{i}}-C_{\tau_{i-1}}\right)^2
\end{align*}
which converges in ucp to $L_4^2[C]=0$. Since $G(C)$ has zero quadratic variation in the strong sense, the integral $\int_0^.Z_-dG(C)$ exists as the ucp limit of $\sum_{\tau^n_i\in P^n,\tau^n_i\le . }Z_{\tau_{i-1}}\left(G(C_{\tau_i})-G(C_{\tau_{i-1}})\right)$. Further manipulation of these Riemann-Stieltjes sums yields,
\begin{align}\label{ZphiC}
\sum_{\tau^n_i\in P^n,\tau^n_i\le . }Z_{\tau_{i-1}}\left(G(C_{\tau_i})-G(C_{\tau_{i-1}})\right)
&= 
\sum_{\tau^n_i\in P^n,\tau^n_i\le . }Z_{\tau_{i-1}}\int_0^1 g(C_{\tau_{i-1}}+\theta(C_{\tau_{i}}-C_{\tau_{i-1}}))d\theta(C_{\tau_{i}}-C_{\tau_{i-1}})\nonumber
\\
&=
\sum_{\tau^n_i\in P^n,\tau^n_i\le . }Z_{\tau_{i-1}} g(C_{\tau_{i-1}})(C_{\tau_{i}}-C_{\tau_{i-1}})\nonumber
\\
&+
\sum_{\tau^n_i\in P^n,\tau^n_i\le . }Z_{\tau_{i-1}}(C_{\tau_{i}}-C_{\tau_{i-1}})\left(\int_0^1 g(C_{\tau_{i-1}}+\theta(C_{\tau_{i}}-C_{\tau_{i-1}}))d\theta-g(C_{\tau_{i-1}})\right).
\end{align}
Let us first study the second term on the right-most side above. Given $\epsilon>0$ let $\delta_1,\delta_2>0$ be so small that $\left|g(u)-g(v)\right|<\epsilon$ for $u,v\in[-L_3,L_3]$ with $|u-v|<\delta_1$ and $\left|C_u-C_v\right|<\delta_1$ whenever $|u-v|<\delta_2$ for $u,v\in[0,t]$. Then if $\n P^n\n<\delta_2$ we get 
$$\left|\int_0^1 g(C_{\tau_{i-1}}+\theta(C_{\tau_{i}}-C_{\tau_{i-1}}))d\theta -g(C_{\tau_{i-1}})\right|
\le
\int_0^1 \left|g(C_{\tau_{i-1}}+\theta(C_{\tau_{i}}-C_{\tau_{i-1}}))-g(C_{\tau_{i-1}})\right|d\theta
\le
\epsilon.$$
Since the ucp limit of $\sum_{\tau^n_i\in P^n,\tau^n_i\le . }Z_{\tau_{i-1}}(C_{\tau_{i}}-C_{\tau_{i-1}})$ is $\int Z_-dC$, we may therefore conclude that the second term on the right-hand side of \eqref{ZphiC} converges to 0, while the right most expression converges in ucp to $\int_0^.Z_-dG(C)$. This implies that the ucp limit $\sum_{\tau^n_i\in P^n,\tau^n_i\le . }Z_{\tau_{i-1}} g(C_{\tau_{i-1}})(C_{\tau_{i}}-C_{\tau_{i-1}})$ has a well-defined limit, $\int_0^.Z_-g(C)dC$ and moreover that $\int_0^.Z_-g(C)dC=\int_0^.Z_-dG(C)$. Therefore $[\int_0^.Z_-g(C)dC]_t=[\int_0^.Z_-dG(C)]_t=0$, as was to be shown.
\end{proof}
Similarly to \cite{NonCont} we will make use of a variant of the $C^2-$ It\^o formula. We may not, however apply the one from \cite{Fol}, as we are working in ucp setting. This actually does not require much innovation as it turns out.
\begin{lemma}\label{Follm}
Suppose $X$ is a process in with quadratic variation in the strong sense and $f$ is a $C^2$ function then
\begin{align*}
f(X_t)=f(X_0)+\sum_{s\le t}\left(\Delta f(X_s)-\Delta X_sf'(X_{s_-}) \right)+\int_0^tf'(X_{s_-})dX_s+\frac 12\int_0^tf''(X_{s_-})d[X]^c_s,
\end{align*}
where $\int_0^tf'(X_{s_-})dX_s$ is a process with quadratic variation in the strong sense and the ucp limit of $\sum_{\tau_i\in P^n}f'(X_{\tau_{i-1}})(X_{\tau_{i}}-X_{\tau_{i-1}})$ where $P^n$ is any sequence of partitions of $[0,t]$ such that $\n P^n\n\to 0$.
\end{lemma}
\begin{proof}
We know that since $f\in C^2$, $f(X)$ has a quadratic variation in the strong sense. Re-arranging, we have that the formula stated in the theorem is equivalent to
$$\int_0^t f'(X_-)dX=f(X_t)-f(X_0)-\frac 12\int_0^tf''(X_{s-})d[X]^c_s-\sum_{s\le t}\left(\Delta f(X_s)-\Delta X_sf'(X_{s_-}) \right). $$
All terms on the right-hand side have quadratic variation in the strong sense (the second one is zero) and therefore so will the left-hand side if we can establish this identity. This follows from the standard arguments in the proof of the It\^o-formula since we do not establish any semimartingale property and we work in the ucp setting.
\end{proof}
\begin{lemma}\label{equiv}
Let $a\in \R^+$, if $\sum_{s\le t} \int_{|x|\le a}|x|\nu(\{s\},dx)<\infty\hspace{2mm} a.s.$ then
$\sum_{s\le t,s\in\mathcal{A}} \int_{|x|\le a}|x|\mu(\{s\},dx)<\infty\hspace{2mm} a.s.$ (where $\mathcal{A}$ is defined in Definition \ref{calA})
and in particular,
$$\left|\sum_{ s\in\mathcal{A}}\int_{|x|\le a}\left(f(X_{s-}+x)-f(X_{s-})-xf'(X_{s-})\right)\mu(\{s\},dx)\right|<\infty \hspace{1mm} a.s. $$
\end{lemma}
\begin{proof}
Let $T_0=0$ and for $k\ge 1$
$$T_k=\inf\{s>T_{k-1}:\nu\left(\{s\},\R\right)>0\}.$$
Then trivially $\{T_k\}_k$ are  (purely) predictable times exhausting $\mathcal{A}$. Similarly we can consider the localizing sequence $U_0=0$, 
$$U_m=\inf\{r>U_{m-1}: \sum_{s\le r} \int_{|x|\le a}|x|\nu(\{s\},dx)\geq m\},$$ 
for $m\ge 1$ which are all predictable times. Since the function $|x|1_{s< U_m}1_{|x|\le a}$ is non-negative and predictable and $T_k$ is a predictable time, for we have 
$$\int_{|x|\le a}|x|1_{T_k< U_m}\nu(\{T_k\},dx)=\E\left[\int_{|x|\le a}|x|1_{T_k< U_m}\mu(\{ T_k\},dx)\ \|\F_{T_k-}\right],$$ 
for any $m,k\ge 0$, due to Property 1.11 of Chapter II in \cite{JACOD}. Taking expectations of both sides above yields
$$\E\left[\int_{|x|\le a}|x|1_{T_k< U_m}\mu(\{T_k\},dx)\right]=\E\left[\int_{|x|\le a}|x|1_{T_k< U_m}\nu(\{ T_k\},dx)\ \right].$$ 
It then follows from monotone convergence that
\begin{align*}
\E\left[\left(\sum_{ s\in\mathcal{A},s\le .}\int_{|x|\le a}|x|\mu(\{s\},dx)\right)_{U_m-}\right]
&=
\E\left[\sum_{ s\in\mathcal{A}}\int_{|x|\le a}|x|1_{s< U_m}\mu(\{s\},dx)\right]
\\
&=
\E\left[\sum_{k\ge 1}\int_{|x|\le a}|x|1_{T_k< U_m}\mu(\{T_k\},dx)\right]
\\
&=
\sum_{k\ge 1}\E\left[\int_{|x|\le a}|x|1_{T_k< U_m}\mu(\{T_k\},dx)\right]
\\
&=
\sum_{k\ge 1}\E\left[\int_{|x|\le a}|x|1_{T_k< U_m}\nu(\{T_k\},dx)\right]
\\
&=
\E\left[\sum_{k\ge 1}\int_{|x|\le a}|x|1_{T_k< U_m}\nu(\{T_k\},dx)\right]
\\
&=\E\left[\left(\sum_{s\le .}\int_{|x|\le a}|x|\nu(\{s\},dx)\right)_{(t\wedge U_m)-}\right]
\le 
m+a,
\end{align*}
which implies $\E\left[\left(\sum_{ s\in\mathcal{A},s\le .}\int_{|x|\le a}|x|\mu(\{s\},dx)\right)_{U_m}\right]\le m+2a$. Therefore, $\sum_{s\in\mathcal{A}}\int_{|x|\le a}|x|\mu(\{s\},dx)<\infty$ on $\{t\le U_m\}$ and since $\P\left(\bigcup_{m\ge 1} \{U_m\ge t\} \right)=1$ we conclude that $\sum_{s\le t}\int_{|x|\le a}|x|\mu(\{s\},dx)<\infty$ a.s.. Letting $S_m=\sup_{x\in[-m,m]}|f'(x)|$ we may also deduce that
\begin{align*}
&\left|\sum_{s\in \mathcal{A}}\int_{|x|\le a}\left(f(X_{s-}+x)-f(X_{s-})-xf'(X_{s-})\right)\mu(\{s\},dx)\right|
\\
\le &\sum_{s\in \mathcal{A}}\int_{|x|\le a}|x|\int_0^1\left|f'(X_{s-}+\theta x)-f'(X_{s-})\right|d\theta\mu(\{s\},dx) \le 2S_m\sum_{s\in \mathcal{A}}\int_{|x|\le a}|x|\mu(\{s\},dx)<\infty,
\end{align*}
a.s..
\end{proof}
Past this point all processes we consider in this manuscript will fulfil $\int|x|d\nu<\infty$ a.s. (i.e. the purely predictable jumps of $X$ have finite variation a.s.).
\begin{thm}\label{absolut}
  Suppose $X$ is a Dirichlet process, that $f$ (and $X$) fulfils the assumption of Lemma \ref{LowLemma} and that $a>0$. Then $f(X)$ is a Dirichlet process (with quadratic variation along the same refining sequence as $X$) that admits to the following decomposition,
$f(X)=Y^a+\Gamma^a$ with
\begin{align}\label{Eq:Ya}
Y^{a}_t&=f(X_0)+\sum_{s\le t}\left(f(X_s)-f(X_{s-})-\Delta X_sf'(X_{s-}) \right)1_{|\Delta X_s|>a}+\int_0^tf'(X_{s-})dZ_s\nonumber
\\
&+\int_0^t\int_{|x|\le a}\left(f(X_{s-}+x)-f(X_{s-})-xf'(X_{s-})\right)(\mu-\nu)(ds,dx)\nonumber
\\
&+\sum_{s\le t} \int_{|x|\le a}\left(f(X_{s-}+x)-f(X_{s-})-xf'(X_{s-})\right)\nu(\{s\},dx), 
\end{align}
where $\mu$ is the jump measure of $X$ with $\nu$ as its compensator, $f'$ is as in Lemma \ref{LowLemma} and $\Gamma^a$ is an adapted continuous process with zero quadratic variation. 
\end{thm}

This theorem is an analogue of Theorem 2.1 in \cite{NonCont}, in the setting of quadratic variations in the strong sense, but with more relaxed assumptions on the transform (\cite{NonCont} considers $C^1$ transforms), the decomposition is exactly of the same type. This type of decomposition will be very useful to us as it gives us a very detailed description of the jumps.
\begin{rema} 
As was shown in \cite{Low}, Assumption \ref{aslow} is always fulfilled by semimartingales. Theorem 1.2. in \cite{Low}, in the time homogeneous case tells us that a semimartingale transformed by a locally Lipschitz function is a sum of a semimartingale and process with zero \textit{continuous} quadratic variation. With the additional assumptions if right-hand derivatives for $f$, we can in this case show that $f(X)$ is the sum of a semimartingale and a continuous process with zero quadratic variation.
\end{rema}
\begin{proof}
Define $\Gamma=f(X)-Y^a$, where the process $\Gamma$ does depend on $a$ but for notational convenience we suppress this in the proof. Note that the jumps of $f(X)$ and $Y^a$ coincide. Indeed, the $\Delta X_sf'(X_{s-})$ jumps in the first sum are cancelled by the jumps exceeding $a$ in the $dZ$ integral, the smaller jumps in the $dZ$ integral are cancelled from the $(\mu-\nu)$ and $\nu$ integrals. Therefore $\Gamma$ is continuous and adapted so the result is true if we can show that $[\Gamma]_t=0$ and that $Y^a$ is a semimartingale. 
\\
\\
If we let $T_m:=\inf\{s:(\Delta X)^*_s\vee X^*_s\vee Z^*_s\ge m\}$ then $\lim_{m\to\infty}\P(T_m> t)=1$. We will also denote $S_m=\sup_{x\in[-m,m]}|f'(x)|$. For any given $\epsilon>0$ we may chose $m$ such that $\P(T_m> t)>1-\epsilon$ we can therefore restrict our attention to $X^{T_m-}$ and for our purpose we can without loss of generality assume that $X=X^{T_m-}$ for some large $m$. Since $f$ has right-hand derivatives everywhere, this implies that $f'$ is Baire class 1 (see for instance Theorem 2 of \cite{Sar}) and we may approximate $f'$ point-wise by a sequence of continuous functions, $\{h_n\}_n$. By considering $h_n\wedge S_m$ (which is still continuous) we may assume that $h_n\le S_m$ on $[-m,m]$. Due to the Weierstrass Theorem we can take a polynomial $p_n$ such that $|p_n(x)-h_n(x)|<1/n$ for every $x\in[-m,m]$. It follows that $p_n\to f'$ point-wise on $[-m,m]$ and that $\sup_{x\in[-m,m]}|p_n(x)|\le S_m+1<\infty$, for all $n$. Moreover note that since $X_-\in\mathbb{L}$ (the space of left-continuous processes with limits from the right) then $p_n(X_-)\in\mathbb{L}$ and so $f'(X_-)=\limn p_n(X_-)\in\mathcal{P}$ ($f'(X_-)$ is predictable). 
\\
\\
We now verify that $Y^a$ is a semimartingale for each $a>0$.  
By Lemma \ref{equiv}, we may expand expand
\begin{align}\label{rewr}
&\int_0^t\int_{|x|\le a}\left(f(X_{s-}+x)-f(X_{s-})-xf'(X_{s-})\right)(\mu-\nu)(ds,dx)=\nonumber
\\
&\int_0^t\int_{|x|\le a}\left(f(X_{s-}+x)-f(X_{s-})-xf'(X_{s-})\right)(\tilde{\mu}-\nu_c)(ds,dx)+\nonumber
\\
&\sum_{s\le t, s\in \mathcal{A}} \int_{|x|\le a}\left(f(X_{s-}+x)-f(X_{s-})-xf'(X_{s-})\right)(\mu-\nu)(\{s\},dx),
\end{align}
where $\tilde{\mu}$ and $\nu_c$ denotes the measures $\mu$ and $\nu$ with all purely predictable jumps (all jumps in $\mathcal{A}$) removed respectively.
Let us now consider Equation \eqref{Eq:Ya} and re-write it using \eqref{rewr} 
\begin{align}\label{newEq}
Y^{a}_t&=f(X_0)+\sum_{s\le t}\left(f(X_s)-f(X_{s-})-\Delta X_sf'(X_{s-}) \right)1_{|\Delta X_s|>a}+\int_0^tf'(X_{s-})dZ_s\nonumber
\\
&+\int_0^t\int_{|x|\le a}\left(f(X_{s-}+x)-f(X_{s-})-xf'(X_{s-})\right)(\tilde{\mu}-\nu_c)(ds,dx)\nonumber
\\
&+\sum_{s\le t, s\in \mathcal{A}} \int_{|x|\le a}\left(f(X_{s-}+x)-f(X_{s-})-xf'(X_{s-})\right)\mu(\{s\},dx).
\end{align}
Our task is now to verify that the right-hand side of \eqref{newEq} is a semimartingale:  
\begin{itemize}
\item
The jump sum $\sum_{s\le t}\left(f(X_s)-f(X_{s-})-\Delta X_sf'(X_{s-}) \right)1_{|\Delta X_s|>a}$ only contains a finite number of jumps for each path and is trivially of finite variation
\item
For the term $\int_0^.f'(X_{s-})dZ_s$, it suffices to show that $(\int_0^.f'(X_{s-})dZ_s)^{T_m}$ is a semimartingale for each $m$ (so that it is a total semimartingale). Note that $f'(X_{s-})$ is predictable and bounded, so clearly the integral is a semimartingale.
\item
The integrand in the $(\tilde{\mu}-\nu_c)$- integral may be re-written as 
$$W(s,x)=\left(f(X_{s-}+x)-f(X_{s-})-xf'(X_{s-})\right)1_{|x|\le a}=x\int_0^1\left(f'(X_{s-}+\theta x)-f'(X_{s-})\right)d\theta1_{|x|\le a}.$$
Using the notation from \cite{JACOD}, 1.24 chapter 2, page 72, we let
$$\widehat{W}_t=\int_{\R}W(t,x)\nu_c(\{t\}\times dx)1_{\int_{\R}|W(t,x)|\nu_c(\{t\}\times dx)<\infty} $$
then it follows that $\widehat{W}_t=a_c(t)=0$ (where $a_c(t)=\nu_c(\{t\}\times \R)$) since $\nu_c$ has no purely predictable jumps by construction. Therefore, if we now also use the notation from page 73 chapter 2 in \cite{JACOD} and denote
$$C(W)_t=\int_{\R}\int_0^t\left(W(s,x)-\widehat{W}_s\right)^2\nu_c(ds,dx)+\sum_{s\le t}(1-a_c(s))\widehat{W}_s^2 $$
then
\begin{align*}
C(W)^{T_m}_t
&=\left(\int_0^t\int_{|x|\le a}(W(s,x)-\widehat{W}_s)^2\nu_c(ds,dx)\right)^{T_m}+0
\\
&=\left(\int_0^t\int_{|x|\le a}W(s,x)^2\nu_c(ds,dx)\right)^{T_m}=\left(\int_0^t\int_{|x|\le a}x^2\left(\int_0^1\left(f'(X_{s-}+\theta x)-f'(X_{s-})\right)d\theta\right)^2\right)^{T_m}
\\
&\le \int_0^{t\wedge T_m}\int_{|x|\le a}x^24S_m^2\nu(ds,dx),
\end{align*}
which is a process with jumps smaller than $4S_m^2a^2$, so after further localization we see that this is locally integrable. Utilizing the fact that $\nu_c$ is the compensator to $\tilde{\mu}$, this implies that (by Theorem 1.33 a) of Chapter II in \cite{JACOD}) \\$\int_0^t\int_{|x|\le a}\left(f(X_{s-}+x)-f(X_{s-})-xf'(X_{s-})\right)(\tilde{\mu}-\nu_c)(ds,dx)$ is a well-defined local martingale.
\item
For the final term we have the following estimate, similarly to Lemma \ref{equiv}
\begin{align*}
&\sum_{s\le t} \int_{|x|\le a}\left|f(X_{s-}+x)-f(X_{s-})-xf'(X_{s-})\right|\nu(\{s\},dx)\le \sum_{s\le t} \int_{|x|\le a}|x|2S_m\nu(\{s\},dx),
\end{align*}
which shows that the series is summable, so it is a pure jump semimartingale.
\end{itemize}

Define $f_n(x)=f(-m)+\int_{-m}^x p_n(u)du$ for $x\in [-m,m]$ so that both $f_n(x)\to f(x)$ and $p_n(x)\to f'(x)$ on $[-m,m]$. By Lemma \ref{Follm} it follows that
$$f_n(X_t)=f_n(X_0)+\sum_{s\le t}\left(\Delta f_n(X_s)-\Delta X_sp_n(X_{s_-}) \right)+\int_0^tp_n(X_{s_-})dX_s+\frac 12\int_0^tp_n'(X_{s_-})d[X,X]^c_s.$$
We now let
\begin{align}\label{fw}
Y^{n,a}_t&=f_n(X_0)+\sum_{s\le t}\left(f_n(X_s)-f_n(X_{s-})-\Delta X_sp_n(X_{s-}) \right)1_{|\Delta X_s|>a}+\int_0^tp_n(X_{s-})dZ_s\nonumber
\\
&+\int_0^t\int_{|x|\le a}\left(f_n(X_{s-}+x)-f_n(X_{s-})-xp_n(X_{s-})\right)(\tilde{\mu}-\nu_c)(ds,dx)\nonumber
\\
&+\sum_{s\le t, s\in \mathcal{A}} \int_{|x|\le a}\left(f_n(X_{s-}+x)-f_n(X_{s-})-xp_n(X_{s-})\right)\mu(\{s\},dx). 
\end{align}
Since each term on the right-hand side of \eqref{newEq} is a semimartingale, so is the sum.
\\
Define $\Gamma^n_t=\int_0^tp_n(X_{s_-})dC_s+\int_0^tp_n'(X_{s_-})d[X,X]^c_s$ ($\Gamma^n_t$ does in fact depend on $a$ but we suppress this dependence for notational convenience since $a$ is fixed at this point), so that $f_n(X)=Y^{n,a}+\Gamma^n$. Clearly $\int_0^tp_n'(X_{s_-})d[X,X]^c_s$ has zero quadratic variation (in the strong sense). As for $\int_0^tp_n(X_{s_-})dC_s$, by Lemma \ref{triangle} it follows that it is sufficient to show that $\int_0^tX_{s_-}^kdC_s=\int_0^t(Z_{s_-}+C_s)^kdC_s$ has zero quadratic variation for any non-negative integer $k$. Furthermore this is equivalent to showing that $\int_0^tZ_{s_-}^{l_1}C_s^{l_2}dC_s$ has zero quadratic variation for any non-negative integers $l_1,l_2$. This follows directly from Lemma \ref{chainrule} by setting $Z'=Z^{l_1}$ and $g(x)=x^{l_2}$. Since $\Gamma^n$ is the sum of two processes with zero quadratic variation, $\Gamma^n$ must also have zero quadratic variation.
We will now show that $[\Gamma]_t=0$. A-priori we do not even know that $[\Gamma]_t$ exists but it will follow from our argument that it both exists and is zero. Let $\{P^k\}_n$ be any sequence of partitions of $[0,t]$ such that $\n P^k\n\to 0$ as $k\to\infty$. We make the following estimate for a fixed $n$,
\\
\\
\begin{align*}
\sqrt{\sum_{\tau_i\in P^k}\left(\Gamma_{\tau_i}-\Gamma_{\tau_{i-1}} \right)^2}&\le\sqrt{\sum_{\tau_i\in P^k}\left((\Gamma_{\tau_i}-\Gamma^n_{\tau_i})-(\Gamma_{\tau_{i-1}}-\Gamma^n_{\tau_{i-1}}) \right)^2}+ \sqrt{\sum_{\tau_i\in P^k}\left(\Gamma^n_{\tau_i}-\Gamma^n_{\tau_{i-1}} \right)^2}
\\
&\le\sqrt{\sum_{\tau_i\in P^k}\left((f(X)_{\tau_i}-f_n(X)_{\tau_i})-(f(X)_{\tau_{i-1}}-f_n(X)_{\tau_{i-1}}) \right)^2} 
\\
&+ \sqrt{\sum_{\tau_i\in P^k}\left((Y^a_{\tau_i}-Y^{n,a}_{\tau_i})-(Y^a_{\tau_{i-1}}-Y^{n,a}_{\tau_{i-1}}) \right)^2} + \sqrt{\sum_{\tau_i\in P^k}\left(\Gamma^n_{\tau_i}-\Gamma^n_{\tau_{i-1}} \right)^2}
\end{align*}
and this shows that
\begin{align}\label{ineqGamma}
\limsup_{k\to\infty}\P\left( \sqrt{\sum_{\tau_i\in P^k}\left(\Gamma_{\tau_i}-\Gamma_{\tau_{i-1}} \right)^2}\ge \epsilon \right)
&\le
\limsup_{k\to\infty}\P\left( \sum_{\tau_i\in P^k}\left((f(X)_{\tau_i}-f_n(X)_{\tau_i})-(f(X)_{\tau_{i-1}}-f_n(X)_{\tau_{i-1}}) \right)^2\ge \frac{\epsilon^2}{3} \right)\nonumber
\\
&+ \limsup_{k\to\infty}\P\left(\sum_{\tau_i\in P^k}\left((Y^a_{\tau_i}-Y^{n,a}_{\tau_i})-(Y^a_{\tau_{i-1}}-Y^{n,a}_{\tau_{i-1}}) \right)^2\ge \frac{\epsilon^2}{3} \right)  
\nonumber
\\
&+ \limsup_{k\to\infty}\P\left(\sum_{\tau_i\in P^k}\left(\Gamma^n_{\tau_i}-\Gamma^n_{\tau_{i-1}} \right)^2\ge \frac{\epsilon^2}{3} \right)\nonumber
\\
&=\P\left(\left[ f(X)-f_n(X) \right]_t \ge \frac{\epsilon^2}{3}\right)+\P\left( \left[ Y^a-Y^{n,a} \right]_t \ge \frac{\epsilon^2}{3}\right)+\P\left( \left[ \Gamma^n \right]_t \ge \frac{\epsilon^2}{3}\right)\nonumber
\\
&=\P\left(\left[ f(X)-f_n(X) \right]_t \ge \frac{\epsilon^2}{3}\right)+\P\left( \left[ Y^a-Y^{n,a} \right]_t \ge \frac{\epsilon^2}{3}\right),
\end{align}
where we used the fact that the map $f-f_n$ fulfils the hypothesis (since $f_n$ is a polynomial) of Lemma \ref{LowLemma}, so it has a quadratic variation, $Y^a-Y^{n,a}$ is a semimartingale so it obviously has a quadratic variation as well and finally we also used the fact that $\Gamma^n$ has zero quadratic variation. To finish the proof we must show that both of the terms on the right-most side of \eqref{ineqGamma} vanishes as $n\to\infty$. For the first term on the right-most side of \eqref{ineqGamma} we have by Lemma \ref{LowLemma} that
\begin{align}\label{fnquad}
\left[ f(X)-f_n(X) \right]_t=\int_0^t \left(f'(X_{s-})-f'_n(X_{s-})\right)^2d[X]^c_s+\sum_{s\le t} \Delta \left(f(X_s)-f_n(X_s)\right)^2.
\end{align}
The second term on the right-hand side of \eqref{fnquad} will vanish a.s. as $n\to\infty$ due to the uniform convergence of $f_n\to f$. The first term will vanish a.s. by dominated convergence. We now tackle the second term on the right-most side of \eqref{ineqGamma}. Since $\tilde{\mu}$ and it's compensator measure $\nu_c$ are void of any predictable jumps it follows from Theorem 1 of chapter 3, section 5 in \cite{Mart}, that if $g(s,x,\omega)$ is locally integrable then 
$$\left[\int_{[0,.]}\int_{|x|\le a} g(s,x,\omega)(\tilde{\mu}-\nu_c)(ds,dx)\right]_t=\int_{[0,t]}\int_{|x|\le a}g(s,x,\omega)^2\tilde{\mu}(ds,dx).$$ 
Using this fact for the second term as well as applying Lemma \ref{triangle}, we make the following estimates,
\begin{align*}
&[Y^a-Y^{n,a}]_t^\frac 12\le\left[\sum_{s\le t}\left(f_n(X_s)-f(X_s)-(f_n(X_{s-})-f(X_{s-}))-\Delta X_s(p_n(X_{s-})-f'(X_{s-})) \right)1_{|\Delta X_s|>a}\right]_t^\frac 12
\\
&+\left[\int_0^.(p_n(X_{s-})-f'(X_{s-}))dZ_s\right]_t^\frac 12
\\
&+\left[\int_0^.\int_{ |x|\le a}\left(f_n(X_{s-}+x)-f(X_{s-}+x)-(f_n(X_{s-})-f(X_{s-}))-x(f'(X_{s-})-p_n(X_{s-}))\right)(\tilde{\mu}-\nu_c)(ds,dx)\right]_t^\frac 12
\\
&+\left[\sum_{s\le .}\int_{ |x|\le a}\left(f_n(X_{s-}+x)-f(X_{s-}+x)-(f_n(X_{s-})-f(X_{s-}))-x(f'(X_{s-})-p_n(X_{s-})\right)\mu(\{s\},dx)\right]^\frac 12_t
\\
&\le \sqrt{\sum_{s\le t} \left((f_n(X_s)-f(X_s))1_{|\Delta X_s|>a}) \right)^2}+\sqrt{\sum_{s\le t} \left((f_n(X_{s_-})-f(X_{s_-}))1_{|\Delta X_s|>a}) \right)^2}
\\
&+\sqrt{\sum_{s\le t} \left(\Delta X_s(p_n(X_{s-})-f'(X_{s-}))1_{|\Delta X_s|>a} \right)^2}
+\sqrt{\int_0^t(p_n(X_{s-})-f'(X_{s-}))^2d[Z]_s}
\\
&+\left(\int_0^t\int_{ |x|\le a}\left(f_n(X_{s-}+x)-f(X_{s-}+x)-(f_n(X_{s-})-f(X_{s-}))-x(f'(X_{s-})-p_n(X_{s-}))\right)^2\tilde{\mu}(ds,dx)\right)^\frac 12
\\
&+\left(\sum_{s\le t, s\in \mathcal{A}}\left(\int_{ |x|\le a}\left(f_n(X_{s-}+x)-f(X_{s-}+x)-(f_n(X_{s-})-f(X_{s-}))-x(f'(X_{s-})-p_n(X_{s-})\right)\mu(\{s\},dx)\right)^2\right)^\frac 12,
\end{align*}

Since the first three sums contain only finitely many jumps a.s., the pointwise convergence ensures pathwise convergence to zero. For the fourth term we use the fact that $(p_n(X_{s-})-f'(X_{s-}))^2\le (2S_m+1)^2$, so we can employ the dominated convergence theorem pathwise (this is a FV process) to conclude that this term vanishes. Now we consider term five. We have the following dominating bound for the integrand,
\begin{align}\label{est}
&\left(f_n(X_{s-}+x)-f(X_{s-}+x)-(f_n(X_{s-})-f(X_{s-}))-x(f'(X_{s-})-p_n(X_{s-}))\right)^2=\nonumber
\\
&\left(x\int_0^1(p_n(X_{s_-}+\theta x)-f'(X_{s_-}+\theta x))d\theta -x(f'(X_{s-})-p_n(X_{s-}))\right)^2\le \nonumber
\\
&x^2\left(\int_0^1(|p_n(X_{s_-}+\theta x)|+S_m)d\theta +(S_m+|p_n(X_{s-})|)\right)^2\le 4x^2\left(2S_m+1\right)^2,
\end{align}
which is $\tilde{\mu}(ds,dx)$-integrable (since $\int_0^t\int_{ |x|\le a}x^2\tilde{\mu}(ds,dx)\le \int_0^t\int_{ |x|\le a}x^2\mu(ds,dx)\le [X]_t$), so we may again apply the dominated convergence theorem (pathwise). For term six we can analogously dominate the integrand by $2|x|(2S+1)$ and apply dominated convergence yet again. 
\end{proof}
By Theorem 1.2 in \cite{Low}, if $f$ is locally Lipschitz continuous then $f(X)$ is the sum of a semimartingale and a process of zero \textit{continuous} quadratic variation in the strong sense. With the aid of Theorem \ref{absolut} and the additional hypothesis of right-hand derivatives we are able to refine this in order to conclude that $f(X)$ is the sum of a semimartingale and process of zero quadratic variation in the strong sense (i.e. a strong Dirichlet process).
\begin{corollary}
Suppose $Z$ is a semimartingale and that $f$ is locally Lipschitz continuous with right-hand derivatives. Then $f(Z)$ is a strong Dirichlet process.
\end{corollary}
\begin{proof}
By Corollary 3.4 in \cite{Low} Assumption \ref{aslow} is fulfilled for all semimartingales. The result now follows directly from Theorem \ref{absolut}.
\end{proof}
\begin{lemma}\label{condilemma}
Suppose $g$ is $\mathbb{B}(C[0,t])\times\mathbb{B}(\mathbb{R})\times\mathbb{B}(\mathbb{R})$-measurable and bounded. Suppose also that $W$ is an adapted continuous process on $[0,t]$, $Y_1$ and $Y_2$ are two real-valued random variables and that $W$ (as a random variable taking values in $C[0,t]$) is independent from $Y_1$ and $Y_2$, then if we let $h(y_1,y_2)=\E\left[g(W,y_1,y_2)\right]$ we have that
$$\E\left[g(W,Y_1,Y_2)\|Y_1,Y_2\right]=h(Y_1,Y_2), \mathsf{a.s.}. $$
\end{lemma}
\begin{proof}
We show that for any $A\in \sigma(Y_1,Y_2)$
$$\E\left[g(W,Y_1,Y_2)1_A\right]=\E\left[h(Y_1,Y_2)1_A\right]. $$
Suppose that $g(w,y_1,y_2)=1_C(w)1_{B_1}(y_1)1_{B_2}(y_2)$, for some $C\in \mathbb{B}\left(C[0,t]\right)$ and $B_1,B_2\in \mathbb{B}(\mathbb{R})$. In this case we get 
$$\E\left[g(W,Y_1,Y_2)1_A\right]=\P(W^{-1}(C))\P\left(Y_1^{-1}(B_1)\cap Y_2^{-1}(B_2)\cap A\right)=\E\left[h(Y_1,Y_2)1_A\right]. $$
Since sets of the form $C\times B_1\times B_2$ form a $\pi$-system generating $\mathbb{B}\left(C[0,t]\right)\times  \mathbb{B}(\mathbb{R})\times  \mathbb{B}(\mathbb{R})$ the result now follows from the monotone class theorem.
\end{proof}
Let $X$ be a cadlag process. We will denote by $\sigma(X)$ the sigma algebra generated by $\left\{X^{-1}(B):B\in\mathcal{D}\right\}$, where $\mathcal{D}$ denotes the Borel sets in $D[0,t]$ (with respect to the Skorokhod topology). The standard notion of independence of two stochastic processes is the following
\begin{defin}
Suppose $X_1,X_2$ are two cadlag processes on $[0,t]$. We say that $X_1$ and $X_2$ are independent if for any $k\in\N$, $t_1,...,t_k\in[0,t]$ and $B_1,B_2\in\mathbb{B}(\R^k)$
$$\P\left(\left(X^1_{t_1},...,X^1_{t_k}\right)\in B_1,\left(X^2_{t_1},...,X^2_{t_k}\right)\in B_2\right)
=
\P\left(\left(X^1_{t_1},...,X^1_{t_k}\right)\in B_1\right)\P\left(\left(X^2_{t_1},...,X^2_{t_k}\right)\in B_2\right) $$
\end{defin}
\begin{lemma}\label{independenceLemma}
If $X^1$ and $X^2$ are two independent cadlag processes, then they are also independent as random variables in $D[0,t]$ (i.e. $\sigma(X^2)$ and $\sigma(X^2)$ are independent).
\end{lemma}
\begin{proof}
Recall that $\mathcal{D}$ is generated by the the following family of sets (see for instance chapter 3, section 12 of \cite{Billy}), 
$$\mathcal{B}=\left\{\pi_{t_1,...,t_k}^{-1}(B): k\in\N, t_1,...,t_k\in [0,t],B\in\mathbb{B}(\mathbb{R}^k\right)\},$$
where $\pi_{t_1,...,t_k}:\mathbb{D}\to\R^k$ denotes the projection $\pi_{t_1,...,t_k}(x)=\left(x_{t_1},...,x_{t_k}\right)$ (for $x\in\mathbb{D}$, the space of cadlag functions). $\mathcal{B}$ is a pi-system that is also closed under complements. We may now introduce the probability measure on $\mathcal{D}$, $\P_X(A)=\P\left(X\in A\right)$ and we now have that for any $A\in\mathcal{D}$ and $\epsilon>0$ there exists a finite union of sets $\bigcup_{l=1}^kC_l$ with $C_l\in\mathcal{B}$ such that $\P\left(A\Delta \bigcup_{l=1^k}C_l\right)<\epsilon$. By definition $C_l=\pi_{t_1,...,t_{k_l}}^{-1}(B_l)$ for some $B_l\in \mathbb{B}(\R^{k_l})$. Rephrasing the independent process assumption implies for $B_1,B_2\in\mathbb{B}(\R^k)$
\begin{align}\label{indep}
\P\left(\left\{\pi_{t_1,...,t_k}(X^1)\in B_1\right\}\bigcap \left\{\pi_{t_1,...,t_k}(X^2)\in B^2\right\}\right)
=
\P\left(\pi_{t_1,...,t_k}(X^1)\in B_1\right)
\P\left( \pi_{t_1,...,t_k}(X^2)\in B_2\right).
\end{align}
Let $F_1\in\sigma(X^1),F_2\in\sigma(X^2)$ then $F_1=\{\omega:X^1(\omega)\in D_1\}$, $F_2=\{\omega:X^2(\omega)\in D_2\}$, for some sets $D_1,D_2\in\mathcal{D}$. Fixing $\epsilon>0$ we now choose $\bigcup_{l=1}^{k_1}C_l^1$, $\bigcup_{l=1}^{k_2}C_l^2$ with $C_l^1,C_l^2\in\mathcal{B}$, $C_l^1=\pi_{t_1^{(1)},...,t_{k_l}^{(1)}}^{-1}(B_l^1)$, $C_l^1=\pi_{t_1^{(2)},...,t_{k_l}^{(2)}}^{-1}(B_l^2)$ for some $B_l^2$ such that $\P_{X^1}\left(D_1\Delta \bigcup_{l=1}^{k_1}C_l^1\right)<\epsilon$, $\P_{X^2}\left(D_2\Delta \bigcup_{l=1}^{k_2}C_l^2\right)<\epsilon$ i.e. $\P\left(X^1\in D_1\Delta \bigcup_{l=1}^{k_1}C_l^1\right)<\epsilon$ and $\P\left(X^2\in D_2\Delta \bigcup_{l=1}^{k_2}C_l^2\right)<\epsilon$. Let $t_1,...,t_k$ be an enumeration of the points $\left(\bigcup_{l=1}^{K^1}\{t_1^{(1)},...,t_{k_l}^{(1)}\}\right)\bigcup\left(\bigcup_{l=1}^{K^2}\{t_1^{(2)},...,t_{k_l}^{(2)}\}\right)$ and let $k$ denote the cardinality of this set. By replacing $B_l^{i}$ ($i=1,2$)with $B_l^i\times\R^{k-k_l}$ whenever $k_l<k$ we have due to \eqref{indep},
\begin{align*}
\P\left(\left\{X^1\in \bigcup_{l=1}^{K_1}C^1_l\right\}\bigcap\left\{X^2\in \bigcup_{l=1}^{K_2}C^2_l\right\}\right)
&=
\P\left(\left\{X^1\in \bigcup_{l=1}^{K_1}\pi_{t_1,...,t_{k}}^{-1}(B_l^1)\right\}\bigcap\left\{X^2\in \bigcup_{l=1}^{K_2}\pi_{t_1,...,t_{k}}^{-1}(B_l^2)\right\}\right)
\\
&=
\P\left(\left\{X^1\in \pi_{t_1,...,t_{k}}^{-1}\left(\bigcup_{l=1}^{K_1}B_l^1\right)\right\}\bigcap\left\{X^2\in \pi_{t_1,...,t_{k}}^{-1}\left(\bigcup_{l=1}^{K_2}B_l^2\right)\right\}\right)
\\
&=
\P\left(\left\{\pi_{t_1,...,t_{k}}(X^1)\in\left(\bigcup_{l=1}^{K_1}B_l^1\right)\right\}\bigcap\left\{\pi_{t_1,...,t_{k}}(X^2)\in\left(\bigcup_{l=1}^{K_2}B_l^2\right)\right\}\right)
\\
&=
\P\left(\pi_{t_1,...,t_{k}}(X^1)\in\left(\bigcup_{l=1}^{K_1}B_l^1\right)\right)
\P\left(\pi_{t_1,...,t_{k}}(X^2)\in\left(\bigcup_{l=1}^{K_2}B_l^2\right)\right)
\end{align*}
Let $F_1=\{X^1\in D_1\}$, $F_2=\{X^2\in D_2\}$, $H_1=\left\{X^1\in \bigcup_{l=1}^{K_1}C^1_l\right\}$ and $H_2=\left\{X^2\in \bigcup_{l=1}^{K_2}C^2_l\right\}$. We now have that
$$\left|\P\left(F_1\cap F_2\right)-\P\left(H_1\cap H_2\right)\right| 
\le
\P\left(\left(F_1\cap F_2\right)\Delta\left(H_1\cap H_2\right)\right). $$
For the right-hand side above we have
\begin{align}\label{OG}
\P\left(\left(F_1\cap F_2\right)\Delta\left(H_1\cap H_2\right)\right)
=
\P\left(\left(F_1\cap F_2\right)\cap\left(H_1^c\cup H_2^c\right)\right)
+
\P\left(\left(H_1\cap H_2\right)\cap\left(F_1^c\cup F_2^c\right)\right).
\end{align}
As
\begin{align*}
\P\left(\left(F_1\cap F_2\right)\cap\left(H_1^c\cup H_2^c\right)\right)
\le
\P\left(F_1\cap F_2\cap H_1^c\right)
+
\P\left(F_1\cap F_2\cap H_2^c\right)
\le
\P\left(F_1\Delta H_1\right)
+
\P\left(F_2\Delta H_2\right)<2\epsilon.
\end{align*}
By an analogous arguments $\P\left(\left(F_1\cap F_2\right)\cap\left(H_1^c\cup H_2^c\right)\right)
\le 2\epsilon$. This leads to
\begin{align}\label{F1F2H1H2}
\left|\P\left(F_1\cap F_2\right)-\P\left(H_1\cap H_2\right)\right| 
\le
\P\left(\left(F_1\cap F_2\right)\Delta\left(H_1\cap H_2\right)\right)<4\epsilon.
\end{align}
Utilizing \eqref{F1F2H1H2} leads to
\begin{align*}
\left|\P\left(F_1\cap F_2\right)-\P\left(F_1\right)\P\left(F_2\right)\right|
&\le
\left|\P\left(F_1\cap F_2\right)-\P\left(H_1\cap H_2\right)\right|
+
\left|\P\left(F_1\right)\P\left(F_2\right)-\P\left(H_1\right)\P\left(H_2\right)\right|
\\
&\le
\left|\P\left(F_1\cap F_2\right)-\P\left(H_1\cap H_2\right)\right|
+
\P\left(F_1\right)\left|\P\left(F_2\right)-\P\left(H_2\right)\right|
+
\P\left(H_2\right)\left|\P\left(F_1\right)-\P\left(H_1\right)\right|
\\
&\le
4\epsilon
+
\left|\P\left(F_2\Delta H_2\right)\right|
+
\left|\P\left(F_1\Delta H_1\right)\right|
\le 6\epsilon.
\end{align*}
Letting $\epsilon\to 0$ gives $\P\left(F_1\cap F_2\right)=\P\left(F_1\right)\P\left(F_2\right)$, which implies that $X^1$ and $X^2$ are independent as random variables.
\end{proof}

\section{Main results}\label{main}


Any Dirichlet process $X=M+V+C$, where $M$ is a local martingale, $V$ is a FV process and $C$ is a process with zero continuous variation can be decomposed as $X=X^c+X^d$, where $X^c=M^c+V^c+C$, $X^d=M^d+V^d$, $M^c$ is the continuous part of $M$, $M^d$ is the purely discontinuous part of $M$, $V^c$ is the continuous part of $V$ and $V^d$ is the jump part of $V$. 
\\
\\
What now follows are the main two results of this manuscript. In our first result we must make an assumption that essentially forbids "hand-picking" jumps depending on the continuous part of the process (this way, one could hand-pick jumps so that we exactly hit the discontinuity points of $f'$). For the same reason we want the distribution of the continuous part of the process as each point $s\le t$ to be non-atomic (i.e. $P(X^c_s=a)=0$ for all $a\in\R$).
\begin{thm}\label{a.s.ettf}
Let $f$ be a primitive function of a cadlag function, assume that $X$ is a Dirichlet process such that $X^c_s$ has a non-atomic distribution for every $s\in[0,t]$ and that the processes $X^c$ and $X^d$ are independent. Let $\{X^n\}_n$ be Dirichlet processes in the strong sense such that $X^n_s\xrightarrow{a.s.}X_s$ for each $s\le t$, $[X^n-X]_t\xrightarrow{a.s.}0$ and that $\{(X^n)^*_t\}_n$ is bounded in probability, then
$$[f(X^n)-f(X)]_t \to 0 \hspace{2mm} a.s..$$
\end{thm}
\begin{rema}
The above theorem can be generalized to functions $f$ that are such that $\limsup_{h\to 0}\left|\frac{f(x+h)-f(x)}{h}\right|<\infty$ and with right-derivatives such that $f'$ has a countable set of discontinuities.
\end{rema}
With the weaker definition of quadratic variation (quadratic variation in the weak sense), we have a similar result, but this time we must strengthen then condition on our transform to be $C^1$.
\begin{thm}\label{C1}
Let $f\in C^1$. Let $\{X^n\}_n$ be Dirichlet processes such that for each $n$, $X^n$ and $X$ have quadratic variations along the same refining sequence and such that $X^n_s\xrightarrow{a.s.}X_s$ for each $s\le t$, $[X^n-X]_t\xrightarrow{a.s.}0$ and that $\{(X^n)^*_t\}_n$ is bounded in probability. Assume also that $X$ and $X^n$ where any predictable jumps is of finite variation. Then
\\$[f(X^n)-f(X)]_t\xrightarrow{a.s.}0$ as $n\to \infty$.
\end{thm}
The proof of Theorem \ref{C1} essentially boils down to a simplified version of the proof of Theorem \ref{a.s.ettf} (which we we give now). After the proof of Theorem \ref{a.s.ettf} we point out the differences in the proof of Theorem \ref{C1}.

\begin{proof}[Proof of Theorem \ref{a.s.ettf}]
As $f$ is the primitive function of a cadlag function, $f'$, it follows from the right continuity of $f'$, that $f$ has right-hand derivatives. Let $a>1$ be some (large) constant to be chosen later. Let $X=Z+C$ where $Z$ is a semimartingale and $C$ has zero quadratic variation. According to Theorem \ref{absolut} we have that $f(X_s)=Y^a_s+\Gamma^a_s$ where $Y^a$ is a semimartingale, $\Gamma^a$ is continuous and $[\Gamma^a]_t=0$ for all $t>0$. We recall that the expression for $Y^a$ is given by
\begin{align}\label{frepr}
Y_t&=f(X_0)+\sum_{s\le t}\left(f(X_s)-f(X_{s-})-\Delta X_sf(X_{s-}) \right)1_{|\Delta X_s|>a}+\int_0^tf'(X_{s-})dZ_s\nonumber
\\
&+\int_0^t\int_{|x|\le a}\left(f(X_{s-}+x)-f(X_{s-})-xf'(X_{s-})\right)(\mu-\nu)(ds,dx)\nonumber
\\
&+\sum_{s\le t} \int_{|x|\le a}\left(f(X_{s-}+x)-f(X_{s-})-xf'(X_{s-})\right)\nu(\{s\},dx). 
\end{align}
By Lemma \ref{equiv}, $\left|\sum_{ s\in\mathcal{A}}\int_{|x|\le a}\left(f(X_{s-}+x)-f(X_{s-})-xf'(X_{s-})\right)\mu(\{s\},dx)\right|<\infty \hspace{1mm} a.s.$, which allows us to re-write,
\begin{align*}
&\int_0^t\int_{|x|\le a}\left(f(X_{s-}+x)-f(X_{s-})-xf'(X_{s-})\right)(\mu-\nu)(ds,dx)=
\\
&\int_0^t\int_{|x|\le a}\left(f(X_{s-}+x)-f(X_{s-})-xf'(X_{s-})\right)(\tilde{\mu}-\nu_c)(ds,dx)+
\\
&\sum_{s\le t, s\in \mathcal{A}} \int_{|x|\le a}\left(f(X_{s-}+x)-f(X_{s-})-xf'(X_{s-})\right)(\mu-\nu)(\{s\},dx),
\end{align*}
where $\tilde{\mu}$ denotes the jump measure $\mu$ with all predictable jumps removed. We may now rewrite \eqref{frepr} as
\begin{align*}
&f(X_0)+\sum_{s\le t}\left(f(X_s)-f(X_{s-})-\Delta X_sf(X_{s-}) \right)1_{|\Delta X_s|>a}+\int_0^tf'(X_{s-})dZ_s\nonumber
\\
&+\int_0^t\int_{|x|\le a}\left(f(X_{s-}+x)-f(X_{s-})-xf'(X_{s-})\right)(\tilde{\mu}-\nu_c)(ds,dx)\nonumber
\\
&+\sum_{s\le t, s\in \mathcal{A}} \int_{|x|\le a}\left(f(X_{s-}+x)-f(X_{s-})-xf'(X_{s-})\right)\mu(\{s\},dx).
\end{align*}
Similarly let $X^n=Z^n+C^n$ where $Z^n$ is a semimartingale and $C^n$ has zero quadratic variation. We again apply Theorem \ref{absolut}, we have that $f(X^n_s)=(Y^n)^a_s+(\Gamma^n)^a_s$ where $(Y^n)^a$ is a semimartingale, $(\Gamma^n)^a$ is continuous and $[(\Gamma^n)^a]_t=0$ for all $t>0$. Arguing as above we see that the expression for $(Y^n)^a$ will be given by
\begin{align*}
Y^n_t&=f(X^n_0)+\sum_{s\le t}\left(f(X^n_s)-f(X^n_{s-})-\Delta X^n_sf(X^n_{s-}) \right)1_{|\Delta X^n_s|>a}+\int_0^tf'(X^n_{s-})dZ^n_s
\\
&+\int_0^t\int_{ |x|\le a}\left(f(X^n_{s-}+x)-f(X^n_{s-})-xf'(X^n_{s-})\right)(\tilde{\mu}_n-(\nu_n)_c)(ds,dx)
\\
&+\sum_{s\le t, s\in \mathcal{A}_n} \int_{ |x|\le a}\left(f(X^n_{s-}+x)-f(X^n_{s-})-xf'(X^n_{s-})\right)\mu_n(\{s\},dx),
\end{align*}
where $\mathcal{A}_n=\left\{(s,\omega): s\le t,\nu_n\left(\omega,\{s\},\R\right)>0\right\}$. Since
\begin{align*}
&f(X^n_{s-}+x)-f(X_{s-}+x)-x\left( f'(X^n_{s-})-f'(X_{s-}) \right)
\\
&=x\left(\int_0^1\left( f'(X^n_{s-}+\theta x)-f'(X_{s-}+\theta x) \right)d\theta  -f'(X^n_{s-})+f'(X_{s-})\right),
\end{align*}
and the term in the parenthesis is clearly locally bounded we can conclude that
\begin{align*}
\int_0^. \int_{|x|\le a}\left( f(X^n_{s-}+x)-f(X_{s-}+x)-x\left( f'(X^n_{s-})-f'(X_{s-}) \right)\right)(\tilde{\mu}-\nu_c)(ds,dx),
\end{align*}
is well defined (as a local martingale). Furthermore, since $X^n$ and $X$ have quadratic variations along the same refining sequence then so does $f(X^n)$ and $f(X)$ by Theorem \ref{absolut} which implies that $(\Gamma^n)^a-\Gamma^a$ has a quadratic variation along this sequence which is zero (the quadratic variation of the semimartingale parts of $f(X^n)$ and $f(X)$ do not depend on the refining sequence). With this in mind we make the following estimate,
\begin{align}\label{g1}
&[f(X^n)-f(X)]^\frac 12_t=[(Y^n)^a-Y^a]^\frac 12_t\le \left[\int f'(X^n_{-})d\left(Z -Z^n\right) \right]_t^{\frac 12}+\left[ \int \left( f'(X^n_{-})-f'(X_{-}) \right)dZ\right]_t^{\frac 12}\nonumber
\\
& +\left[\int_0^. \int_{|x|\le a}\left(f(X^n_{s-}+x)-f(X_{s-}+x)+f(X^n_{s-}))-f(X_{s-}))-x( f'(X^n_{s-})-f'(X_{s-}) ) \right)(\tilde{\mu}-\nu_c)(ds,dx)\right]_t^{\frac 12}\nonumber
\\
&+\left[\int_0^. \int_{|x|\le a}\left(\left(f(X^n_{s-}+x)-f(X^n_{s-}))\right)-xf'(X^n_{s-}) \right)(\tilde{\mu}-\nu_c)(ds,dx)\right.\nonumber
\\
%
&\left.\left.-\int_0^. \int_{|x|\le a}\left(\left(f(X^n_{s-}+x)-f(X^n_{s-}))\right)-xf'(X^n_{s-}) \right)(\tilde{\mu}_n-(\nu_n)_c)(ds,dx)\right]_t^{\frac 12}\right.\nonumber
\\
%
&+ \left[\sum_\mathcal{A} \int_{|x|\le a}\left((f(X^n_{s-}+x)-f(X_{s-}+x)-f(X^n_{s-}))+f(X_{s-}))-x( f'(X^n_{s-})-f'(X_{s-}) )\right)\mu(\{s\},dx) \right]_t^{\frac 12}\nonumber
\\
%
&+\left[\sum_{s\le t, s\in \mathcal{A}} \int_{|x|\le a}\left(f(X^n_{s-}+x)-f(X^n_{s-})-xf'(X^n_{s-})\right)(\mu)(\{s\},dx)
\right.\nonumber
\\
&\left.-\sum_{s\le t, s\in \mathcal{A}_n} \int_{|x|\le a}\left(f(X^n_{s-}+x)-f(X^n_{s-})-xf'(X^n_{s-})\right)(\mu_n)(\{s\},dx)  \right]_t^{\frac 12}\nonumber
\\
&+\left[\sum_{s\le t}\left(f(X^n_s)-f(X^n_{s-})-\Delta X^n_sf'(X^n_{s-}) \right)1_{|\Delta X^n_s|>a}-\sum_{s\le t}\left(f(X_s)-f(X_{s-})-\Delta X_sf'(X_{s-}) \right)1_{|\Delta X_s|>a}\right]_t^\frac 12,
\end{align}
where we substituted the expressions  for $Y$, $Y^a$, did a bit of rearrangement and then used Lemma \ref{triangle}. 
\\
$\tilde{\mu}_n$ and $\tilde{\mu}_n$ are void of predictable jumps so the quadratic variation of the $\tilde{\mu}-\nu_c$ and $\tilde{\mu}_n-\nu_c$-integrals can be computed in the same manner as in the proof of Theorem \ref{absolut}. We now expand \eqref{g1} as

\begin{align}\label{grisny2}
&[f(X^n)-f(X)]^{\frac 12}_t\le \left( \int_{0^+}^t  f'(X^n_{-})^2d\left[Z-Z^n\right]_s \right)^{\frac 12}+\left( \int_{0^+}^t \left(  f'(X^n_{s_{-}})-f'(X_{s_{-}})\right)^2d[Z]_s \right)^{\frac 12} \nonumber
\\
& +\left(\int_0^t\int_{|x|\le a}\left((f(X^n_{s-}+x)-f(X_{s-}+x)-f(X^n_{s-}))+f(X_{s-}))-x( f'(X^n_{s-})-f'(X_{s-}) )\right)^2\tilde{\mu}(ds,dx)\right)^{\frac 12}\nonumber
\\
&+\left[\int_0^. \int_{|x|\le a}\left(\left(f(X^n_{s-}+x)-f(X^n_{s-}))\right)-xf'(X^n_{s-}) \right)((\tilde{\mu}-\nu_c)(ds,dx)\right.\nonumber
\\
&\left.\left.-\int_0^. \int_{|x|\le a}\left(\left(f(X^n_{s-}+x)-f(X^n_{s-}))\right)-xf'(X^n_{s-}) \right)(\tilde{\mu}_n-(\nu_n)_c)(ds,dx)\right]_t^{\frac 12}\right.\nonumber
\\
& +\left(\sum_{s\le t,s\in\mathcal{A}} \left(\int_{|x|\le a}\left(f(X^n_{s-}+x)-f(X_{s-}+x)-f(X^n_{s-})+f(X_{s-})-x(f'(X^n_{s-})-f'(X_{s-}))\right)\mu(\{s\},dx) \right)^2 \right)^{\frac 12} \nonumber
\\
&+\left[\sum_{s\le t, s\in \mathcal{A}} \int_{|x|\le a}\left(f(X^n_{s-}+x)-f(X^n_{s-})-xf'(X^n_{s-})\right)(\mu)(\{s\},dx)
-
\right.\nonumber
\\
&\left.\sum_{s\le t, s\in \mathcal{A}_n} \int_{|x|\le a}\left(f(X^n_{s-}+x)-f(X^n_{s-})-xf'(X^n_{s-})\right)(\mu_n)(\{s\},dx)  \right]_t^{\frac 12}\nonumber
\\
&+\left[\sum_{s\le t}\left(f(X^n_s)-f(X^n_{s-})-\Delta X^n_sf'(X^n_{s-}) \right)1_{|\Delta X^n_s|>a}-\sum_{s\le t}\left(f(X_s)-f(X_{s-})-\Delta X_sf'(X_{s-}) \right)1_{|\Delta X_s|>a}\right]_t^\frac 12,
\end{align}
here we used the fact that for term five and six of \eqref{g1} (corresponding to term five and six of \eqref{grisny2}), the expressions inside the quadratic variations are quadratic pure jump semimartingales so their contributions are just the square sums of their jumps. 
First we note that for any stopping time $T\le t$, $|\Delta(X^n-X)_T|\le \sqrt{[X^n-X]_T}$ so $X^n_{s-}\xrightarrow{a.s.}X_{s-}$. If we let $B_R=\left\{\sup_n(X^n)^*_t\vee X^*_t\le R\right\}$ then $\P\left(\cup_{R} B_R\right)=\lim_{R\to\infty}\P(B_R)=1$ so we may assume that $X,\Delta X$, $\{X^n\}_n$ and $\{\Delta X^n\}_n$ are all uniformly bounded by the constant $R$.  We will let $M=\sup_{x\in[-R-a,R+a]} |f'(x)|$ (which must be finite since $f$ is locally Lipschitz).
So it will suffice to show that each term tends to zero a.s.. If we choose $a>R$ large then term seven in \eqref{grisny2} is identically zero on $B_R$, so will assume that term seven is identically zero.
\\
For the first term in \eqref{grisny2}
$$\left(\int_{0^+}^t f'(X^n_{-})^2d\left[Z -Z^n\right]_s\right)^\frac 12\le M [Z^n-Z]_t^\frac 12=M [X^n-X]_t^\frac 12, $$
which converges to zero a.s. by assumption.
\\
\\
We move on to the second term in \eqref{grisny2}. Denote $\mathsf{disc}(f')$ as the set of discontinuity points for $f'$ and let $y\in\R$ be arbitrary. We note that since $f'$ is cadlag, $\mathsf{disc}(f')$ is at most countable. By stopping we may assume that $[X^c]_t\le L$. For any sequence of fixed-time partitions $\{P^n\}_n$ with mesh tending to zero we have the following a.s. point-wise limit,
$$\int_0^t1_{\{X_s=y\}}d[X]^c_s=\lim_{n\to\infty}\sum_{t_i\in P^n, i<|P^n|}1_{\{X_{t_i}^d+X_{t_i}^c=y\}}\left([X^c]_{t_{i+1}}-[X^c]_{t_{i}}\right). $$
Since $X^c\rightarrow X^c_s$ is a continuous and therefore a measurable map this implies that if we let \\$S_n(X^c)_s=\sum_{t_i\in P^n, t_i\le s}\left(X^c_{t_i}-X^c_{t_{i-1}}\right)^2$ then $X^c\rightarrow S_n(X^c)_s$ is also a measurable map. As $S_n(X^c)_s\xrightarrow{\P}[X^c]_s$, there exists a subsequence $\{n_k\}_k$ such that $S_{n_k}(X^c)_s\xrightarrow{a.s.}[X^c]_s$. This implies that $X^c\rightarrow[X^c]_s$ is a measurable map and therefore $g(X^c,a)=1_{\{X_{t_i}^c+a=y\}}\left([X^c]_{t_{i+1}}-[X^c]_{t_{i}}\right)$ is a positive measurable map.
Since 
$$\sum_{t_i\in P^n, i<|P^n|}1_{\{X_{t_i}^d+X_{t_i}^c=y\}}\left([X^c]_{t_{i+1}}-[X^c]_{t_{i}}\right)\le [X^c]_t\le L,$$ 
we have by dominated convergence that
\begin{align*}
\E\left[\int_0^t1_{\{X_s=y\}}d[X]^c_s\|X^d\right]
&=
\E\left[\lim_{n\to\infty}\sum_{t_i\in P^n, i<|P^n|}1_{\{X_{t_i}^d+X_{t_i}^c=y\}}\left([X^c]_{t_{i+1}}-[X^c]_{t_{i}}\right)\|X^d\right]
\\
&=
\lim_{n\to\infty}\sum_{t_i\in P^n, i<|P^n|}\E\left[1_{\{X_{t_i}^d+X_{t_i}^c=y\}}\left([X^c]_{t_{i+1}}-[X^c]_{t_{i}}\right)\|X^d\right].
\end{align*}
Since $X^c$ is independent as a random variable in $C[0,t]$ from $X^d$ as random variable in $D[0,t]$ (and therefore also from $X^d_s$ for any $s\in [0,t]$), by applying Lemma \ref{condilemma} with $g$ defined as above (formally the lemma uses three arguments but we can just use a dummy constant as the third argument) we get that 
$$\E\left[1_{\{X_{t_i}^d+X_{t_i}^c=y\}}\left([X^c]_{t_{i+1}}-[X^c]_{t_{i}}\right)\|X^d\right]= h_i(X^d)$$
where $h_i(a)=\E\left[1_{\{a+X_{t_i}^c=y\}}\left([X^c]_{t_{i+1}}-[X^c]_{t_{i}}\right)\right]$. By applying Hölders inequality and the fact that \\$\left([X^c]_{t_{i+1}}-[X^c]_{t_{i}}\right)^2\le [X^c]^2_t\le L^2$ we have,
$$h_i(a)\le 
\E\left[\left([X^c]_{t_{i+1}}-[X^c]_{t_{i}}\right)^2\right]^{\frac 12}\P\left(X_{t_i}^c+a=y\right)^{\frac 12}
\le
L\P\left(X_{t_i}^c=y-a\right)^{\frac 12}
=0 ,$$
for all $a\in\R$. This implies
$$\E\left[\int_0^t1_{\{X_s=y\}}d[X]^c_s\|X^d\right]=\lim_{n\to\infty}\sum_{t_i\in P^n, i<|P^n|}h_i(X^c)=0\textsf{ a.s.}. $$
Taking expectation yields
$$\E\left[\int_0^t1_{\{X_s=y\}}d[X]^c_s\right]=\E\left[\E\left[\int_0^t1_{\{X_s=y\}}d[X]^c_s\|X^d\right]\right]=0.  $$
As $\int_0^t1_{\{X_s=y\}}d[X]^c_s\ge 0$, this implies $\int_0^t1_{\{X_s=y\}}d[X]^c_s=0$ a.s.. If we now let $A_m$ be a finite subset of $\R$ with $m$ elements, $\{a_1,...,a_m\}$, then
$$\int_0^t1_{\{s:X_s\in A_m\}}d[X]^c_s=\sum_{k=1}^m\int_0^t1_{\{s:X_s= a_k\}}d[X]^c_s=0.$$
Let $\{a_k\}_k$ be an enumeration of the countable set $\mathsf{disc}(f')$ and let $A_m$ denote its first $m$ elements. It follows from monotone convergence (path-wise) that
\begin{align}\label{cont}
\int_0^t1_{\{s:X_s\in \mathsf{disc}(f') \}}d[X]^c_s
=
\lim_{m\to\infty}\int_0^t1_{\{s:X_s\in A_m\}}d[X]^c_s
=0.
\end{align}
Hence, since if $s$ is a continuity point of $f'$ then $f$ is differentiable in $s$ and therefore $f$ in conjunction with $X$ fulfils Assumption \ref{aslow}. 
\\
\\
Let $\{T_k\}_k$ be an enumeration of the jumps of $X$, let $T_1^b=\inf\{s>0:|\Delta X_s|\ge b\}$ and for $k>1$, $T_k^b=\inf\{s>T_{k-1}:|\Delta X_s|\ge b\}$. Consider the map on $C[0,t]\times \R\times\R$, $(W,s,u)\mapsto W_s+u$, this map is clearly continuous and therefore measurable on $\mathbb{B}(C[0,t]\times\R\times\R)$. By separability of $C[0,t]$ and $\R$ we have $\mathbb{B}(C[0,t]\times\R\times\R)=\mathbb{B}(C[0,t])\times\mathbb{B}(\mathbb{R})\times\mathbb{B}(\mathbb{R})$. As a result, the function $1_{W_s+u=a}$ is also $\mathbb{B}(C[0,t])\times\mathbb{B}(\mathbb{R})\times\mathbb{B}(\mathbb{R})$-measurable. Let $\{P^n\}_n$ is a sequence of (fixed time) partitions with mesh tending to zero and define 
$$H_1^n(\tilde{X})=\min\left\{t_i\in P^n: \left|\tilde{X}_{t_i}-\tilde{X}_{t_{i-1}}\right|\ge b \right\}1_{\exists t_i\in P^n:\left|\tilde{X}_{t_i}-\tilde{X}_{t_{i-1}}\right|\ge b}+\infty 1_{\not\exists t_i\in P^n:\left|\tilde{X}_{t_i}-\tilde{X}_{t_{i-1}}\right|\ge b},$$ 
for any cadlag process $\tilde{X}\in D[0,t]$, then clearly $H_1^n$ is $\mathbb{B}(D[0,t])$-measurable and therefore so is $T_1^b(\tilde{X})=\lim_{n\to\infty}H_1^n(\tilde{X})$. Therefore $T_1^b$ is a $\sigma(X^d)$-measurable random variable. Since $X^d$ as a random variable in $D[0,t]$ is independent of $X^c$, as a random variable in $C[0,t]$ (by Lemma \ref{independenceLemma}), so is $T_1^b$. For $j>1$ we may analogously define 
\begin{align*}
H_j^n(\tilde{X})&=\min\left\{t_i\in P^n, t_i>H_{j-1}^n(\tilde{X}): \left|\tilde{X}_{t_i}-\tilde{X}_{t_{i-1}}\right|\ge b \right\}1_{\exists t_i\in P^n:\left|\tilde{X}_{t_i}-\tilde{X}_{t_{i-1}}\right|\ge b, t_i>H_{j-1}^n}
\\
&+\infty 1_{\not\exists t_i\in P^n:\left|\tilde{X}_{t_i}-\tilde{X}_{t_{i-1}}\right|\ge b,t_i>H_{j-1}^n},
\end{align*}
so that similarly to before $T_j^b(\tilde{X})=\lim_{n\to\infty}H_j^n(\tilde{X})$ and so $\{T_k^b\}_k$ are all independent of $X^c$. We now let $W=X^c$, $Y_1=T_k^b$, $Y_2=X^d_{T_k^b-}$ in Lemma \ref{condilemma} and get that 
$$\E\left[1_{X_{T_k^b-}=a}\|T_k^b\right]=\E\left[1_{X_{T_k^b}^c+X_{T_k^b-}^d=a}\|T_k^b\right]=h(X_{T_k-}^d,T_k^b) \hspace{1mm}\mathsf{a.s.},$$
where 
$$h(y_1,y_2)=\E\left[1_{X_{y_1}^c+y_2=a}\right]=\P\left(X_{y_1}^c+y_2=a\right)=0,$$ 
since $X_{y_1}^c$ has a continuous distribution. Thus, $\E\left[1_{X_{T_k^b-}^b=a}\|T_k^b\right]=0$ a.s. and therefore 
$$\P\left(X_{T_k^b-}=a\right) =\E\left[\E\left[1_{X_{T_k^b-}=a}\|T_k^b\right]\right]=0.$$
We may now conclude that $\P\left(X_{T_k^b-}\in\mathsf{disc}(f') \right)=0$, for every $k$ since $\mathsf{disc}(f')$ is countable. Next, for any $\epsilon>0$ we may now choose $b>0$ so small that 
$$\sum_{s\le t: |\Delta X_s|< b}|\Delta X_s|^2=\sum_{k: |\Delta X_{T_k}|< b}|\Delta X_{T_k}|^2=\sum_{k}|\Delta X_{T_k}|^2-\sum_{k}|\Delta X_{T_k^b}|^2<\epsilon.$$
Utilizing \eqref{cont}, we now re-write the expression inside the square root of the second term in \eqref{grisny2} as
\begin{align}\label{frst}
\int_{0^+}^t \left(  f'(X^n_{s_{-}})-f'(X_{s_{-}})\right)^2d[Z]_s
&=\int_{0^+}^t \left(  f'(X^n_{s_{-}})-f'(X_{s_{-}})\right)^2d[X]^c_s+\sum_{k>0}(\Delta X_{T_k})^2\left(f'(X^n_{T_k-}) -f'(X_{T_k-})\right)^2\nonumber
\\
&=\int_{0^+}^t \left(  f'(X^n_{s_{-}})-f'(X_{s_{-}})\right)^21_{\{s:X_s\not\in \mathsf{disc}(f')\}}d[X]^c_s\nonumber
\\
&+
\int_{0^+}^t \left(  f'(X^n_{s_{-}})-f'(X_{s_{-}})\right)^21_{\{s:X_s\not\in \mathsf{disc}(f')\}}d[X]^c_s\nonumber
\\
&+\sum_{k>0: |\Delta X_{T_k}|<b}(\Delta X_{T_k})^2\left(f'(X^n_{T_k-}) -f'(X_{T_k-})\right)^2
\nonumber
\\
&+
\sum_{k>0}(\Delta X_{T_k^b})^2\left(f'(X^n_{T_k^b-}) -f'(X_{T_k^b-})\right)^2\nonumber
\\
&\le \int_{0^+}^t \left(  f'(X^n_{s_{-}})-f'(X_{s_{-}})\right)^21_{\{s:X_s\not\in \mathsf{disc}(f')\}}d[X]^c_s
\nonumber
\\
&+
\sum_{k>0}(\Delta X_{T_k^b})^2\left(f'(X^n_{T_k^b-}) -f'(X_{T_k^b-})\right)^2+M
2\epsilon
\end{align}
On $\{s:X_s\not\in \mathsf{disc}(f')\}$, $X^n_{s-}\xrightarrow{a.s.}X_{s-}$ implies $f'(X^n_{s-})\xrightarrow{a.s.}f'(X_{s-})$. Since $\left(  f'(X^n_{s_{-}})-f'(X_{s_{-}})\right)^2\le M^2$, it follows from the dominated convergence theorem, applied pathwise a.s. that 
$$\int_{0^+}^t \left(  f'(X^n_{s_{-}})-f'(X_{s_{-}})\right)^21_{\{s:X_s\not\in \mathsf{disc}(f')\}}d[X]^c_s\to 0 $$
a.s. as $n\to \infty$.
For the second term of \eqref{frst} we first note that the terms inside the sum all converge to zero since $X_{T_k^b-}$, for each $k$, is a.s. a continuity point for $f'$. Since this sum only contain a finite number of terms, and each term converges we get that $\sum_{k>0}(\Delta X_{T_k^b})^2\left(f'(X^n_{T_k^b-}) -f'(X_{T_k^b-})\right)^2\to 0$ as $n\to\infty$ a.s.. Therefore $\lim_{n\to\infty}\int_{0^+}^t \left(  f'(X^n_{s_{-}})-f'(X_{s_{-}})\right)^2d[Z]_s<M^2\epsilon$ a.s., since $\epsilon$ is arbitrary we conclude that the second term in \eqref{grisny2} vanishes.
\\
\\
Recall that the third term or \eqref{grisny2} is given by
\begin{align}\label{tre}
\int_0^t\int_{|x|\le a}\left(f(X^n_{s-}+x)-f(X_{s-}+x)-f(X^n_{s-})+f(X_{s-})-x(f'(X^n_{s-})-f'(X_{s-}))\right)^2\tilde{\mu}(ds,dx).
\end{align}
Since $\mu$ is supported on the set $\left\{\Delta X\not=0\right\}$, so is $\tilde{\mu}$ (which is supported on a subset of the same set), hence we may re-write \eqref{tre} as
\begin{align}\label{treAsSum}
\sum_{k=1}^\infty&\int_{|x|\le a}\left(f(X^n_{T_k-}+x)-f(X_{T_k-}+x)-f(X^n_{T_k-})+f(X_{T_k-})-x(f'(X^n_{T_k-})-f'(X_{T_k-}))\right)^2\tilde{\mu}(\{T_k\},dx)\nonumber
\\
&\le 
2\sum_{k=1}^\infty\int_{|x|\le a}\left(f(X^n_{T_k-}+x)-f(X_{T_k-}+x)-f(X^n_{T_k-})+f(X_{T_k-})\right)^2\tilde{\mu}(\{T_k\},dx)
\nonumber
\\
&+
2\sum_{k>0: |\Delta X_{T_k}|<b}(\Delta X_{T_k})^2\left(f'(X^n_{T_k-}) -f'(X_{T_k-})\right)^2
+
2\sum_{k>0}(\Delta X_{T_k^b})^2\left(f'(X^n_{T_k^b-}) -f'(X_{T_k^b-})\right)^2.
\end{align}
Since $X_{T_k^b-}$ is a.s. a continuity point of $f'$ the integrand converges pointwise to zero. Note that 
\begin{align}\label{fdiff}
\left|f(X^n_{s-}+x)-f(X_{s-}+x)-f(X^n_{s-})+f(X_{s-})\right|=|x|\left|\int_0^1\left( f'(X^n_{s-}+\theta x)-f'(X_{s-}+\theta x) \right)d\theta\right|\le 2M|x|
\end{align} 
and  $\left|x(f'(X^n_{s-})-f'(X_{s-}))\right|\le 2M|x|$
on the set $B_R$. We conclude that the whole integrand in \eqref{tre} can be dominated by $16M^2x^2$, which is $\tilde{\mu}$-integrable and since the integrand converges pointwise to zero, \eqref{tre} will vanish due to dominated convergence applied path-wise. For the two sums on the right-hand side of \eqref{fdiff} we already showed in \eqref{frst} that these terms tend to zero as $b\to 0$.
\\
\\
We now proceed with term four. For any $r>0$, we have that on $\bigcap_{k=1}^\infty\left\{\exists s\le t: |\Delta X_s|\in \left[r\left(1-\frac{1}{2^k}\right),r\left(1-\frac{1}{2^{k+1}}\right)\right)\right\}$,
\begin{align*}
\infty
=
\frac{r}{2}\sum_{k=1}^\infty 1
&=\frac{r}{2}\sum_{k=1}^\infty 1_{|\Delta X_s|\in\left[r\left(1-\frac{1}{2^k}\right),r\left(1-\frac{1}{2^{k+1}}\right)\right)} 
\\
&\le \sum_{s\le t}\left(\Delta X_s\right)^2 \le [X]_t,
\end{align*}
but $[X]_t<\infty$ a.s. and therefore $\P\left(\bigcap_{k=1}^\infty\left\{\exists s\le t: |\Delta X_s|\in \left[r\left(1-\frac{1}{2^k}\right),r\left(1-\frac{1}{2^{k+1}}\right)\right)\right\}\right)=0$. In other words, for any $r>0$, $\P\left(\bigcup_{k=1}^\infty A_k(r)\right)=1$, where $A_k(r)=\left\{\not\exists s\le t: |\Delta X_s|\in \left[r\left(1-\frac{1}{2^k}\right),r\left(1-\frac{1}{2^{k+1}}\right)\right)\right\}$. Let us now define $\delta(k,r)=\frac{r}{2^{k+1}}$ (we shall suppress the dependence on $r$ and $k$ from now on)  and $L(r,k)=r\left(1-\frac{3}{2^{k+2}}\right)$ (we will suppress the dependence on $r$ and from now on) so that $|\Delta X_s|\not\in \left(L(r)-\delta,L(r)+\delta\right)$, $\forall s\le t$ on $A_k(r)$. On $A_k(r)$, if for some stopping time $T\le t$, $|\Delta X_T|<L(r)$ then $|\Delta X_T|<L(r)-\delta$ and similarly $|\Delta X_T|> L(r)$ implies $|\Delta X_T|> L(r)+\delta$. Since $|\Delta (X^n-X)_T|\xrightarrow{a.s.}0$ for any stopping time $T\le t$ it follows that if $\tilde{T}_1$ and $\tilde{T}_2$ are stopping times such that $|\Delta X_{\tilde{T}_1}|\ge L(r)$ and $|\Delta X_{\tilde{T}_2}|< L(r)$ there exists some $N_{r,\delta}(\omega)\in\N$ such that $n\ge N_{r,\delta}(\omega)$ implies that $|\Delta X^n_{\tilde{T}_1}|\ge L(r)$ and $|\Delta X^n_{\tilde{T}_2}|< L(r)$. Next,
\begin{align}\label{jumptrick}
&\sum_{s\le t} (\Delta X^n_s)^21_{|\Delta X^n_s|\le L(r)}=\sum_{s\le t} (\Delta (X^n-X)_s+\Delta X_s)^21_{|\Delta X^n_s|\le L(r)} \nonumber
\\
\le &2\sum_{s\le t} (\Delta (X^n-X)_s)^2+2\sum_{s\le t} (\Delta X_s)^21_{|\Delta X^n_s|\le L(r)}\nonumber
\\
\le &2\sum_{s\le t} (\Delta (X^n-X)_s)^2+2\sum_{s\le t} (\Delta X_s)^21_{|\Delta X_s|\le L(r)}\le 2[X^n-X]_t+2\sum_{s\le t} (\Delta X_s)^21_{|\Delta X_s|\le L(r)},
\end{align}
for any $n\ge N_{r,\delta}$.
In any bounded region that is also bounded away from zero we may split both the $(\tilde{\mu}-(\nu)_c)$- and $(\tilde{\mu}_n-(\nu_n)_c)$ integrals appearing in term four of \eqref{grisny2} into $(\tilde{\mu}-(\mu)_n)$ and $(\tilde{\mu}-\nu_c)$ integrals. By Lemma \ref{triangle} and a little bit of rearrangement,
\begin{align}\label{trm5}
&\left[\int_0^. \int_{|x|\le a}\left(\left(f(X^n_{s-}+x)-f(X^n_{s-}))\right)-xf'(X^n_{s-}) \right)((\tilde{\mu}-(\nu)_c)(ds,dx)\right.\nonumber
\\
&\left.-\int_0^. \int_{|x|\le a}\left(\left(f(X^n_{s-}+x)-f(X^n_{s-}))\right)-xf'(X^n_{s-}) \right)(\tilde{\mu}_n-(\nu_n)_c)(ds,dx)\right]_t^{\frac 12}\nonumber
\\
&\le \left[\int_0^.\int_{|x|\le L(r)}\left(\left(f(X^n_{s-}+x)-f(X^n_{s-}))\right)-xf'(X^n_{s-}) \right)((\tilde{\mu}-\nu_c)(ds,dx)\right]_t^{\frac 12}\nonumber
\\
&+\left[\int_0^.\int_{|x|\le L(r)}\left(\left(f(X^n_{s-}+x)-f(X^n_{s-}))\right)-xf'(X^n_{s-}) \right)(\tilde{\mu}_n-(\nu_n)_c)(ds,dx)\right]_t^{\frac 12}\nonumber
\\
&+\left[\int_0^.\int_{L(r)<|x|\le a}\left(\left(f(X^n_{s-}+x)-f(X^n_{s-}))\right)-xf'(X^n_{s-}) \right)(\tilde{\mu}-\tilde{\mu}_n)(ds,dx)\right]_t^{\frac 12}\nonumber
\\
&+\left[\int_0^.\int_{L(r)<|x|\le a}\left(\left(f(X^n_{s-}+x)-f(X^n_{s-}))\right)-xf'(X^n_{s-}) \right)(\nu-\nu_n)_c(ds,dx)\right]_t^{\frac 12},
\end{align}
where the final term is zero since the process inside the quadratic variation is continuous with finite variation. For the first term on the right-hand side of \eqref{trm5} we have
\begin{align*}
&\left[\int_0^.\int_{|x|\le L(r)}\left(\left(f(X^n_{s-}+x)-f(X^n_{s-}))\right)-xf'(X^n_{s-}) \right)((\tilde{\mu}-\nu_c)(ds,dx)\right]_t^{\frac 12}
\\
&=\left(\int_0^t\int_{|x|\le L(r)}\left(\left(f(X^n_{s-}+x)-f(X^n_{s-}))\right)-xf'(X^n_{s-}) \right)^2\tilde{\mu}(ds,dx)\right)^{\frac 12}
\\
&\le\left(\int_0^t\int_{|x|\le L(r)}4M^2x^2\tilde{\mu}(ds,dx)\right)^{\frac 12}\le
 2M \left(\sum_{s\le t} (\Delta X_s)^21_{|\Delta X_s|\le L(r)} \right)^{\frac 12}
\end{align*}
which can be arbitrarily small by choosing $r$ small enough.

The second term of \eqref{trm5} equals 
\begin{align*}
&\left(\left[\int_0^.\int_{|x|\le L(r)}(((f(X^n_{s-}+x)-f(X^n_{s-})))-xf'(X^n_{s-}) )(\tilde{\mu}_n-(\nu_n)_c)(ds,dx)\right]_t\right)^\frac 12
\\
\le &\left(\int_0^.\int_{|x|\le L(r)}x^24M^2\tilde{\mu}_n(ds,dx)\right)^\frac 12
\le 2M\left(\sum_{s\le t} (\Delta X^n_s)^21_{|\Delta X^n_s|\le L(r)}\right)^\frac 12.
\end{align*}
According to \eqref{jumptrick}, 
\begin{align}\label{asd}
&\sum_{s\le t} (\Delta X^n_s)^21_{|\Delta X^n_s|\le L(r)}\le 2[X^n-X]_t+2\sum_{s\le t} (\Delta X_s)^21_{|\Delta X_s|\le L(r)},
\end{align}
for $n\ge N_{r,\delta}$.
For any $\epsilon>0$, we may pick $N(\epsilon)$ such that $[X^n-X]_t<\epsilon$ for $n\ge N(\epsilon)$. This implies that $\lim_{r\to 0^+}\sum_{s\le t} (\Delta X^n_s)^21_{|\Delta X^n_s|\le L(r)}<\epsilon$ for $n\ge N(\epsilon)\vee N_{r,\delta}$. In other words the limit as $r\to 0^+$ of the second term of \eqref{trm5} is uniformly bounded above by $2M\epsilon$ over $n\ge N(\epsilon)$.
Next, for the third term in the right-hand side of \eqref{trm5}  we can rewrite it as
\begin{align*}
&\left[\sum_{s\le ., s\not\in \mathcal{A}\cup\mathcal{A}_n} \left(f(X^n_{s-}+\Delta X_s)-f(X^n_{s-}+\Delta X^n_s)- \Delta (X^n-X)_sf'(X^n_{s-})\right) 1_{L(r)<|\Delta X_s|\le a}\right]_t
\\
&\le \sum_{s\le t}\left(f(X^n_{s-}+\Delta X_s)-f(X^n_{s-}+\Delta X^n_s)- \Delta (X^n-X)_sf'(X^n_{s-})\right)^2 1_{L(r)<|\Delta X_s|\le a},
\end{align*}
which contains finitely many terms, each of which converges to zero a.s.. Thus the limit of \eqref{trm5} as $n\to\infty$ is bounded above by $2M\epsilon$ for any $\epsilon>0$, letting $\epsilon\to 0^+$ shows that \eqref{trm5} goes to zero.
\\
\\
  We now handle term five of \eqref{grisny2}. Since $\{T_k\}_k$ exhausts all jumps we have that $\{U_k: |\Delta X_{U_k}|>0\}\subseteq \{T_k\}_k$ and therefore
\begin{align}\label{nastsist}
&\sum_{ s\in\mathcal{A}} \left(\int_{|x|\le a}\left(f(X^n_{s-}+x)-f(X_{s-}+x)-f(X^n_{s-})+f(X_{s-})-x(f'(X^n_{s-})-f'(X_{s-}))\right)\mu(\{s\},dx) \right)^2 \nonumber
\\
\le
&\sum_{k=1}^\infty \left(\int\left|f(X^n_{U_k-}+x)-f(X_{U_k-}+x)-f(X^n_{U_k-})+f(X_{U_k-})-x(f'(X^n_{U_k-})-f'(X_{U_k-}))\right|\mu(\{U_k\},dx) \right)^2 \nonumber
\\
\le
&\sum_{k=1}^\infty\left(\int \left|f(X^n_{T_k-}+x)-f(X_{T_k-}+x)-f(X^n_{T_k-})+f(X_{T_k-})-x(f'(X^n_{T_k-})-f'(X_{T_k-}))\right|\mu(\{T_k\},dx)\right)^2\nonumber
\\
=
&\sum_{k=1}^\infty\left( \left( f(X^n_{T_k}+\Delta X_{T_k})-f(X^n_{T_k})-f( X_{T_k}+\Delta X_{T_k})+f( X_{T_k})-\Delta X_{T_k}(f'(X^n_{T_k-})-f'(X_{T_k-}))\right)\right)^2\nonumber
\\
\le
&2\sum_{k=1}^\infty  \left(f(X^n_{T_k}+\Delta X_{T_k})-f(X^n_{T_k})-f( X_{T_k}+\Delta X_{T_k})+f( X_{T_k})\right)^2
+
2\sum_{k=1}^\infty \Delta X_{T_k}^2(f'(X^n_{T_k-})-f'(X_{T_k-}))^2.
\end{align}
Due to \eqref{fdiff},
\begin{align*}
\left(f(X^n_{T_k}+\Delta X_{T_k})-f(X^n_{T_k})-f( X_{T_k}+\Delta X_{T_k})+f( X_{T_k})\right)^2
\le
8M^2 |\Delta X_{T_k}|^2
\end{align*}
and since
$\sum_{k=1}^\infty |\Delta X_{T_k}|^2\le [X]_t<\infty $ it follows from dominated convergence (applied path-wise with respect to the counting measure) that the first term on the right-most side of \eqref{nastsist} converges to zero. The second term was already shown to converge to zero a.s. in \eqref{fdiff}.
\\
\\
Let us now consider term six of \eqref{grisny2}. We first fix $r>0$ so small that $\sum_{s\le t,s\in \mathcal{A}}|\Delta X_s|1_{|\Delta X_s|> L(r)}$ is small enough for our purposes and then fix some small $\delta>0$. Let as before $\{U_k\}_k$ be a exhausting sequence of (predictable) stopping times for $\mathcal{A}$ and similarly we let $\{U^n_k\}_k$ be a exhausting sequence of (predictable) stopping times for $\mathcal{A}_n$. We shall now show that on $A_\delta(r)$ then $n\ge N_{r,\delta}(\omega)$ (as defined above \eqref{jumptrick}) implies
\begin{align}\label{A}
\bigcup_{k:|\Delta X^n_{U^n_k}|>L(r)}\left\{U^n_k \right\} = \bigcup_{k:|\Delta X_{U_k}|>L(r)}\left\{U_k \right\}.
\end{align}
Note that the above two sets are random, where $\bigcup_{k:|\Delta X_{U_k}|>L(r)}\left\{U_k \right\}=\mathcal{A}\cap\left\{s:|\Delta X_s|>L(r)\right\}$ and \\$\bigcup_{k:|\Delta X^n_{U^n_k}|>L(r)}\left\{U^n_k \right\}=\mathcal{A}_n\cap\left\{s:|\Delta X^n_s|>L(r)\right\}$. We now show \eqref{A}. Since $U_k$ is predictable for each $k$, and $1_{\nu^n\left(\{s\},\R\right)=0}|x|$ is a predictable function, we have 
\begin{align*}
0=\int |x|1_{\nu^n\left(\{U_k\},\R\right)=0}\nu^n\left(\{U_k\},dx\right)
&=
\E\left[\int |x|1_{\nu^n\left(\{U_k\},\R\right)=0}\nu^n\left(\{U_k\},dx\right)\|\F_{U_k-}\right]
\\
&=
\E\left[1_{\nu^n\left(\{U_k\},\R\right)=0}|\Delta X^n_{U_k}|\|\F_{U_k-}\right]
\end{align*}
a.s. which implies $1_{\nu^n\left(\{U_k\},\R\right)=0}|\Delta X^n_{U_k}\|=0$ a.s.. Therefore $\P\left(\left\{\nu^n\left(\{U_k\},\R\right)=0\right\}\cap\left\{|\Delta X^n_{U_k}|>0\right\}\right)=0$ for all $n$ and $k$. On $A_\delta(r)$, if $n\ge N_{r,\delta}(\omega)$ then $|\Delta X_{U_k}|>L(r)$ if and only if $|\Delta X^n_{U_k}|>L(r)$. So if $|\Delta X_{U_k}|>L(r)$ this a.s. implies $\nu^n\left(\{U_k\},\R\right)>0$ and therefore 
$$
\bigcup_{k:|\Delta X_{U_k}|>L(r)}\left\{U_k \right\}
\subseteq 
\bigcup_{k:|\Delta X^n_{U^n_k}|>L(r)}\left\{U^n_k \right\} .
$$
A completely analogous argument shows the reverse inclusion.
We now compute the following bound for term six of \eqref{grisny2},
\begin{align}\label{muterms}
&\left[\sum_{s\in \mathcal{A}} \int_{|x|\le a}\left(f(X^n_{s-}+x)-f(X^n_{s-})-xf'(X^n_{s-})\right)(\mu)(\{s\},dx)
\right.\nonumber
\\
&\left.-\sum_{s\in \mathcal{A}_n} \int_{|x|\le a}\left(f(X^n_{s-}+x)-f(X^n_{s-})-xf'(X^n_{s-})\right)(\mu_n)(\{s\},dx)  \right]_t^{\frac 12}\le\nonumber
\\
&\left[\sum_{s\in\mathcal{A}}\int_{|x|\le L(r)}\left(f(X^n_{s-}+x)-f(X^n_{s-})-xf'(X_{s-})\right)(\mu)(\{s\},dx)\right]^{\frac 12}_t
\nonumber
\\
&+\left[\sum_{s\in\mathcal{A}_n}\int_{|x|\le L(r)}\left(f(X^n_{s-}+x)-f(X^n_{s-})-xf'(X_{s-})\right)(\mu_n)(\{s\},dx)\right]^{\frac 12}_t\nonumber
\\
&+\left[\sum_{s\le t, s\in \mathcal{A}} \int_{L(r)<|x|\le a}\left(f(X^n_{s-}+x)-f(X^n_{s-})-xf'(X^n_{s-})\right)(\mu)(\{s\},dx)
\right.\nonumber
\\
&\left.-\sum_{ s\in \mathcal{A}_n} \int_{L(r)<|x|\le a}\left(f(X^n_{s-}+x)-f(X^n_{s-})-xf'(X^n_{s-})\right)(\mu_n)(\{s\},dx)  \right]_t^{\frac 12}\nonumber
\\
&\le\left(\sum_{s\in\mathcal{A}}\int_{|x|\le L(r)}\left(4M^2x^2\right)(\mu)(\{s\},dx)\right)^{\frac 12}
+
\left(\sum_{s\in\mathcal{A}_n}\int_{|x|\le L(r)}\left(4M^2x^2\right)(\mu_n)(\{s\},dx)\right)^{\frac 12}\nonumber
\\
&+\left[\sum_{ s\in \mathcal{A}} \int_{L(r)<|x|\le a}\left(f(X^n_{s-}+x)-f(X^n_{s-})-xf'(X^n_{s-})\right)(\mu)(\{s\},dx)
\right.\nonumber
\\
&\left.-\sum_{ s\in \mathcal{A}_n} \int_{L(r)<|x|\le a}\left(f(X^n_{s-}+x)-f(X^n_{s-})-xf'(X^n_{s-})\right)(\mu_n)(\{s\},dx)  \right]_t^{\frac 12}
\end{align}
The first and second term on the right-most side above are bounded by $ 2M \left(\sum_{s\le t} (\Delta X_s)^21_{|\Delta X_s|\le L(r)} \right)^{\frac 12}$ and $ 2M \left(\sum_{s\le t} (\Delta X^n_s)^21_{|\Delta X^n_s|\le L(r)} \right)^{\frac 12}$ respectively. Both of these terms tend to zero (uniformly over $n$ for the second term) as $r\to 0^+$, due to \eqref{jumptrick}. Since
\begin{align*}
&\sum_{ s\in \mathcal{A}} \int_{L(r)<|x|\le a}\left(f(X^n_{s-}+x)-f(X^n_{s-})-xf'(X^n_{s-})\right)(\mu)(\{s\},dx)
\\
&-
\sum_{ s\in \mathcal{A}_n} \int_{L(r)<|x|\le a}\left(f(X^n_{s-}+x)-f(X^n_{s-})-xf'(X^n_{s-})\right)(\mu_n)(\{s\},dx)
\\
&=
\sum_{ k:|\Delta X_{ U_k}|>L(r)}\left(\left(\left(f(X^n_{U_k-}+\Delta X_{U_k})-f(X^n_{U_k-})\right)-\Delta X_{U_k}f'(X^n_{U_k-})\right)\right)
\\
&-
\sum_{ k:|\Delta X^n_{ U^n_k}|>L(r)}\left(\left(\left(f(X^n_{U^n_k-}+\Delta X^n_{U^n_k})-f(X^n_{U^n_k-})\right)-\Delta X^n_{U^n_k}f'(X^n_{U_k-})\right)\right).
\end{align*}
Therefore, for $n\ge N_{r,\delta}(\omega)$ on $A_\delta(r)$ the final term equals (due to \eqref{A}),
\begin{align*}
\left(\sum_{k: L(r)<|\Delta X_{U_k}|\le a }\left((f(X^n_{U_k-}+\Delta X_{U_k}))-(f(X^n_{U_k-}+\Delta X^n_{U_k}))-(\Delta X_{U_k}-\Delta X^n_{U_k})f'(X^n_{U_k-})\right)^2\right)^{\frac 12}
\end{align*}
which converges pathwise a.s. on $A_\delta(r)$ to zero due to the fact that this sum contains finitely many terms, $\Delta X^n_{U_k}\xrightarrow{a.s.}\Delta X_{U_k}$ for each $k$, and the fact that $X_{T_k}$ and $X_{T_k-}$ are a.s. continuity points (recall that $\{U_k: |\Delta X_{U_k}|>0\}\subseteq \{T_k\}_k$), for each $k$.
\end{proof}

\begin{proof}[Comments on the proof of Theorem \ref{C1}]
Note that we can't apply Theorem \ref{absolut} as it requires quadratic variation in the strong sense. Instead, we may apply Theorem 2.1 in \cite{NonCont}, note that this decomposition however cuts the jump sum in \eqref{Eq:Ya} at 1 instead of $a$, but this choice is completely arbitrary and the Theorem is still true if we use $a$ as cut-off point.
None of the arguments regarding discontinuity points in the proof of Theorem \ref{a.s.ettf} are required since $f\in C^1$, we also do not need to concern ourselves with fulfilling Assumption \ref{aslow}. All assumptions of Theorem 2.1. in \ref{aslow} are trivially fulfilled. The rest of the proof is then identical to the one of Theorem \ref{a.s.ettf}.
\end{proof}
\begin{proof}[Proof of Proposition \ref{PropVar}]
Using the decomposition \eqref{FolDe}, Lemma \ref{triangle} and the fact that the first and final terms in \eqref{FolDe} have zero quadratic variation we get
\begin{align}\label{basic}
[f(X^n)-f(X)]^{\frac 12}_t
&\le 
\left( \int_{0^+}^t  f'(X^n_{-})^2d\left[X-X^n\right]_s \right)^{\frac 12}+\left( \int_{0^+}^t \left(  f'(X^n_{s_{-}})-f'(X_{s_{-}})\right)^2d[X]_s \right)^{\frac 12} \nonumber
\\
&+\left[ \sum_{s\le t}\Delta \left(f(X_s) -f(X^n_s)\right)\right]^{\frac 12}
+\left[ \sum_{s\le t}\Delta X_s\left(f'(X_s) -f'(X^n_s)\right)\right]^{\frac 12}
+\left[ \sum_{s\le t}\left(\Delta X_s-\Delta X^n_s\right)f'(X^n_{s-})\right]^{\frac 12}.
\end{align}
The first two terms can be dealt with just as in the proof of \ref{a.s.ettf}. For the third term, letting $L(r)$ be as in the proof of Theorem \ref{a.s.ettf},
\begin{align*}
\left[ \sum_{s\le t}\Delta \left(f(X_s) -f(X^n_s)\right)\right]
&=\sum_{s\le t}
\left(\Delta \left(f(X_s) -f(X^n_s)\right)\right)^2
\\
&\le 2\sum_{s\le t}
\left(\Delta f(X_s) \right)^21_{|\Delta X_s|\le L(r)}
+
2\sum_{s\le t}
\left(\Delta f(X^n_s) \right)^21_{|\Delta X_s|\le L(r)}
\\
&+
\sum_{s\le t}\left(\Delta \left(f(X_s) -f(X^n_s)\right)\right)^21_{|\Delta X_s|> L(r)}.
\end{align*}
Since
$$ \sum_{s\le t}
\left(\Delta f(X^n_s) \right)^21_{|\Delta X_s|\le L(r)}
=
\sum_{s\le t} (\Delta X^n_s)^2\left(\int_0^1 f'\left(X^n_s +\theta\Delta X^n_s\right)d\theta\right)^21_{|\Delta X_s|\le L(r)},$$
by a localization argument and the fact that $\{(X^n)^*_t\}_n$ is bounded in probability one may readily deduce that $\left(\int_0^1 f'\left(X^n_s +\theta\Delta X^n_s\right)d\theta\right)^2\le M$, for some $M\in\R^+$. Employing the same trick as in \eqref{jumptrick} now shows that we may make the term $ \sum_{s\le t}
\left(\Delta f(X^n_s) \right)^21_{|\Delta X_s|\le L(r)}$ uniformly small over all $n$. Similarly for some $M'\in\R^+$, we may bound
$$\sum_{s\le t} \left(\Delta f(X^n_s) \right)^21_{|\Delta X_s|\le L(r)}
\le  M'\sum_{s\le t}\left(\Delta X_s\right)^21_{|\Delta X_s|\le L(r)},$$
which goes to zero as $r\to 0$. The sum $\sum_{s\le t}\left(\Delta \left(f(X_s) -f(X^n_s)\right)\right)^21_{|\Delta X_s|> L(r)}$ only contains a finite of terms and by the assumption on point wise convergence, this sum vanishes. For the fourth term of \eqref{basic},
\begin{align*}
\left[ \sum_{s\le t}\Delta X_s\left(f'(X_s) -f'(X^n_s)\right)\right]
&=
\sum_{s\le t}(\Delta X_s)^2\left(f'(X_s) -f'(X^n_s)\right)^2
\\
&\le \sum_{s\le t}(\Delta X_s)^2 1_{|\Delta X_s|<r}2(M'+M)
+
\sum_{s\le t}(\Delta X_s)^2\left(f'(X_s) -f'(X^n_s)\right)^21_{|\Delta X_s|<r}.
\end{align*}
For every $r>0$ the second sum above contains a finite number of terms that will vanish as $n\to\infty$. Letting $r\to 0$ makes the first term vanish. For the fifth and final term of \eqref{basic}, we have
\begin{align*}
\left[ \sum_{s\le t}\left(\Delta X_s-\Delta X^n_s\right)f'(X^n_{s-})\right] &=
\sum_{s\le t} \left(\Delta\left( X^n-X\right)\right)^2f'(X^n_{s-})^2
\\
&\le
\sum_{s\le t} M \left(\Delta\left( X^n-X\right)\right)^2
\le M[X^n-X]_t,
\end{align*}
which converges to zero.
\end{proof}

\hfill


\begin{thebibliography}{99}

\bibitem[P. Billingsley. (1999)]{Billy}
Billingsley, P. (1999).
Convergence of probability measures.
\textit{Wiley Series in Probability and Statistics: Probability and Statistics John Wiley and Sons Inc., New York, Second edition}
\bibitem[Coquet, M\'emin and S{\l}omi{\'n}ski(2003)]{NonCont}
Coquet, F., M\'emin, J. and S{\l}omi{\'n}ski, L. (2003).
On Non-Continuous Dirichlet Processes.
\textit{Journal of Theoretical Probability}, {\bf 16}, 197.
\bibitem[F\"ollmer(1981)]{Fol}
F\"ollmer, H. (1981). Dirichlet processes. In \textit{Lecture Notes in Maths.}, Vol. {\bf 851}, pp. 476-–478, Springer-Verlag, Berlin/Heidelberg/New York.
\bibitem[F\"ollmer(1981)]{FolIto}
  F\"ollmer, H. (1981). Calcul d'Ito sans probabilit\'es. In \textit{S\'eminaire de probabilit\'es (Strasbourg).}, Vol. {\bf 851}, , pp. 143-150, Springer-Verlag, Berlin/Heidelberg/New York.
\bibitem[Jacod and Shiryaev(2003)]{JACOD}
Jacod, J. and Shiryaev, A. N. (2003). 
\textit{Limit Theorems for Stochastic Processes}. Springer-Verlag.
\bibitem[Liptser and Shiryaev(1986)]{Mart}
Liptser, R. and Shiryaev, A. (1986) \textit{Theory of Martingales}. Springer-Verlag.
\bibitem[Lowther(2010)]{Low}
Lowther, G. (2010). Nondifferentiable functions of one-dimensional semimartingales. \textit{Ann. Prob.} {\bf 38} (1) 76 - 101.
\bibitem[Protter(1992)]{PRT}
Protter, P. (1992).	
\textit{Stochastic Integration and Differential Equations}. Springer-Verlag, Berlin, Heidelberg, Second edition.
\bibitem[Sarkhel(2003)]{Sar} Sarkhel, D. N. (2003). Baire one functions. \textit{Bulletin-Institute of Mathematics Academia Sinica}, {\bf 31}(2), 143-149.
\bibitem[Stricker(1988)]{Stricker}
Stricker, C. (1988).
Variation conditionnelle des processus stochastiques.
\textit{Ann. Inst. H. Poincaré Probab. Statist.}, {\bf 24},  295–305.
\end {thebibliography}
\end{document}